\documentclass[final]{elsarticle}
\usepackage{amsmath,graphicx,latexsym,amssymb, amscd, psfrag,verbatim, amsfonts,amsthm}
\usepackage{amsfonts,bm,mathrsfs,subfigure,mathtools}
\usepackage{algorithm,amsxtra,color}
\usepackage[algo2e,ruled,vlined]{algorithm2e}
\usepackage{graphicx}
\usepackage{setspace}
\usepackage{tabularx}
\usepackage{enumerate}
\usepackage{tikz}
\usepackage{pgfplots}
\usepackage{empheq}
\usepackage[makeroom]{cancel}
\pgfplotsset{compat=1.11}
\usepackage{float}
\usepackage{booktabs}
\usepackage{algorithmicx,algpseudocode}
\usepackage[top=1in, bottom=1in, left=0.7in, right=1.25in]{geometry}
\usepackage{multirow}
\usepackage[pdftex, unicode=true,
colorlinks,
linkcolor=red,
anchorcolor=blue,
citecolor=green
]{hyperref}

\newtheorem{thm}{Theorem}[section]
\newtheorem{lem}[thm]{Lemma}
\newtheorem{cor}[thm]{Corollary}

\newtheorem{remark}[thm]{Remark}

\newcommand{\lb}{\left (}
\newcommand{\rb}{\right )}

\newcommand{\bff}{\bm f}

\newcommand{\bw}{\bm w}

\newcommand{\bx}{\bm x}

\newcommand{\bX}{\bm X}

\newcommand{\bu}{\bm u}

\newcommand{\bv}{\bm v}

\newcommand{\bxi}{\bm \xi}

\newcommand{\totalnorm}[1]{{\left\vert\kern-0.25ex\left\vert\kern-0.25ex\left\vert #1
		\right\vert\kern-0.25ex\right\vert\kern-0.25ex\right\vert}}

\begin{document}
	\begin{frontmatter}
		\title{
		%Fully discrete parameter-robust error estimates of a grad-div stabilized Crank-Nicolson artificial compressibility method 
		Fully discrete parameter-robust error analysis of a grad-div stabilized Crank-Nicolson artificial compressibility method for the Navier-Stokes equations}
		\author[ouc]{Feiyu Chen}
		\ead{feiyuchen@stu.ouc.edu.cn}
		\author[usc]{Lili Ju}
		\ead{ju@math.sc.edu}
		\author[ouc,lmm]{Rihui Lan\corref{cor}}
\ead{lanrihui@ouc.edu.cn}
		\author[ouc,lmm]{Shusen Xie}
\ead{shusenxie@ouc.edu.cn}
\address[ouc]{School of Mathematical Sciences, Ocean University of China, Qingdao, Shandong 266100, China}
\address[lmm]{Laboratory of Marine Mathematics, Ocean University of China, Qingdao, Shandong 266100, China}
\address[usc]{Department of Mathematics, University of South Carolina, Columbia, SC 29208, USA}
\cortext[cor]{Corresponding author}
		
		\begin{abstract}
			The artificial compressibility method (ACM) relaxes the incompressibility constraint in the Navier-Stokes equations by introducing a perturbation term proportional to the time derivative of the pressure, scaled by a small positive parameter $\varepsilon$. Consequently, the ACM system inherently involves two small parameters: the fluid viscosity $\nu$ and the artificial compressibility parameter $\varepsilon$. While existing temporal analyses of ACM provide insights into its behavior, rigorous fully discrete error estimates that are robust with respect to both parameters remain a significant gap in the literature. In this paper, we propose a second-order Crank-Nicolson fully discrete ACM scheme and establish its parameter-robust optimal error estimates. To enhance stability and ensure robustness, we incorporate two distinct grad-div stabilization strategies: one facilitates the decoupling of velocity and pressure computations, while the other guarantees robustness with respect to both $\nu$ and $\varepsilon$. For spatial discretization, we employ the Scott-Vogelius finite element pair, which is crucial for the decoupling and error analysis. The resulting parameter-uniform bounds are crucial for ensuring the long-time accuracy of the scheme, circumventing the exponential dependence on the Reynolds number typically introduced by Gr\"onwall's lemma. Numerical experiments are provided to validate the theoretical findings and demonstrate the efficiency of the proposed methods.
		\end{abstract}
	\end{frontmatter}
	
	\section{Introduction}
In this paper,  we consider  the following non-stationary incompressible Navier-Stokes (NS) equations  in the primitive variables:
\begin{equation}\label{NS:orig}
	\left\{
	\begin{aligned}
		\bu_t+\lb\bu\cdot\nabla\rb\bu-\nu\Delta \bu + \nabla p&=\bm f, &&\text{ in } (0, T]\times\Omega, \\
		\nabla\cdot\bu&=0, &&\text{ in }(0, T]\times\Omega,
	\end{aligned}
	\right.
\end{equation}
subject to the Dirichlet boundary condition  $\bu(t,\bx)  = \bu_D(\bx, t)$ on $(0, T]\times\partial\Omega$ and 
the initial condition $\bu(0,\bx)  =\bu_0(\bx)$  in $\Omega,$
where $\Omega$ denotes an open bounded and connected domain in ${\mathbb R}^d$ ($d=2 \text{ or }3$), $T>0$ is the terminal time, $\bu=\bu(t,\bx)$ is the fluid velocity, $p=p(t,\bx)$ is the fluid pressure, $\nu=\frac{1}{Re}$ is the viscosity with  ${Re}$ denoting the Reynolds number, and $\bm f$ is the external body force. 
Designing efficient, accurate, and long-term stable numerical methods for solving the above incompressible NS equations \eqref{NS:orig} in scientific and engineering applications remains a highly demanding and challenging task

First of all, the accurate numerical simulation of  flows with large Reynolds numbers presents persistent theoretical and practical challenges. As Kolmogorov's K41 theory \cite{Kolmogorov} suggests, adequately resolving the eddies in such flows requires a mesh size scaling as $\approx Re^{-3/4}$ in each direction, often leading to prohibitive computational costs. In practice, however, small viscosities and dominant convective effects—hallmarks of high-Reynolds-number regimes—make the construction of robust numerical methods essential. Robustness, in this context, refers to error estimates whose constants do not depend on inverse powers of the viscosity, allowing stable simulations on relatively coarse grids \cite{garcia2021convergence}. Many existing error analyses \cite{huang2021stability,belding2022efficient,heywood1990finite} for fully discrete Navier–Stokes formulations remain viscosity-dependent, limiting their applicability to moderate Reynolds numbers. Furthermore, the frequent use of Gronwall’s lemma \cite{gronwall1919note,bellman1943stability}  in stability analysis tends to introduce exponential dependence on the Reynolds number via the Gronwall constant, thereby precluding robust long-time estimates. Notably, as observed in prior work, achieving robust error bounds for classical finite element pairs such as Taylor–Hood typically requires special treatment of the nonlinear term, for instance through the EMAC formulation \cite{olshanskii2020longer}, which yields a viscosity-independent Gronwall constant and ensures long-time stability. Complementing such formulations, the addition of grad-div stabilization \cite{olshanskii2002low,case2011connection,lan2026robust} is widely recognized to enhance numerical stability by improving local mass conservation, particularly in high-Reynolds-number regimes. Recent advances, summarized in contemporary review article \cite{garcia2021convergence}, have demonstrated the feasibility of obtaining viscosity-robust error estimates for velocity, marking important progress toward reliable coarse-grid simulations of turbulent flows.

Secondly, the incompressibility condition, $\nabla\cdot\mathbf{u}=0$, poses a foundational challenge in designing stable and efficient numerical schemes. This divergence-free constraint is stringent and, if not correctly handled, can degrade the quality of numerical solutions \cite{john2017divergence,cappanera2025artificial}. In traditional finite element methods, stability requires the discrete satisfaction of the inf-sup condition (also called as LBB condition) \cite{boffi2013mixed}. While widely-implemented elements like the Taylor-Hood pair \cite{taylor1973numerical} are popular for their ease of implementation, they fail to ensure local mass conservation. To address such issue, divergence-free elements like the Scott-Vogelius (SV) type \cite{scott1985conforming,scott1985norm} are designed to offer local conservation. However, it requires some special mesh splitting, like Alfeld splitting and Powell-Sabin splitting.  

As an alternative class of methods, pressure-projection methods \cite{chorin1968numerical,guermond2006overview} decouple the velocity and pressure solves by relaxing the incompressibility constraint in a predictor step, followed by a projection to enforce incompressibility. However, these methods inherently introduce splitting errors. For instance, projection methods can generate spurious pressure oscillations due to approximate boundary conditions. To overcome this issue, Quarteroni et al. \cite{quarteroni2000factorization} proposed the algebraic splitting (also known as Yosida method) to decouple the velocity and pressure unknowns based on the discretized linear system, where the boundary conditions are properly handled before splitting. Some related work and  further developments can be found in \cite{saleri2005pressure,rebholz2017improved,rebholz2019efficient}.

The Artificial Compressibility Method (ACM) emerges as a distinct and significant paradigm, pioneered by Chorin \cite{chorin1997numerical} and Teman \cite{temam1969approximation1,temam1969approximation2}. Its core innovation is the regularization of the strict incompressibility constraint by introducing a pseudo-time derivative of pressure via a penalty parameter $\varepsilon > 0$:
\begin{equation}\label{acm:rog}
	\varepsilon \frac{\partial p}{\partial t} + \nabla\cdot \bu = 0.
\end{equation}
This transforms the governing equations from an elliptic-parabolic system into a hyperbolic-parabolic type. 
This transformation enables the application of efficient time-marching schemes to solve the resulting system in pseudo time  \cite{hejranfar2018application}.
The significance and advantages of ACM are multifaceted:
\begin{itemize}
	\item Superior Computational Complexity: A key advantage of ACM over projection methods is its avoidance of solving a full pressure Poisson equation at each step, whose computational cost scales with $h^{-2}$ (where $h$ is the mesh size) and can become a computational bottleneck \cite{cappanera2025artificial}. Instead, the linear systems in AC formulations can scale more favorably, for instance, with $O(\tau h^{-2})$, offering a potential performance benefit \cite{lundgren2023high,guermond2015high}.
\item Framework for High-Order Accuracy: Contrary to the historical limitation of first-order temporal accuracy, recent breakthroughs have demonstrated that ACM provides a robust framework for constructing high-order accurate numerical schemes in time \cite{guermond2015high,guermond2019high}. This addresses a primary drawback of classical projection methods and places modern ACM on par with other high-order approaches.
\item Versatility and Solver Compatibility: By recasting the problem into a hyperbolic-parabolic form, ACM allows the application of a wide array of sophisticated time-integration and upwind schemes from the compressible flow arena \cite{hejranfar2018application}. Furthermore, its formulation accommodates various spatial discretizations, including finite element, finite volume, and finite difference methods \cite{lundgren2023high}.
\end{itemize}

Despite its significant advantages, ACM faces persistent analytical and practical challenges. Two key issues are: (i) ensuring robust performance across a wide spectrum of Reynolds numbers, and (ii) rigorously controlling the perturbation parameter $\varepsilon$ within the error analysis to prevent order reduction. The latter is paramount for establishing a scheme's practical reliability. A rigorous numerical analysis must be constructed to avoid the emergence of an $\varepsilon^{-1}$ term in the final error estimate. The presence of such a term in an a priori bound induces severe, often catastrophic, order reduction as $\varepsilon \to 0$ (with $\varepsilon$ typically chosen as $\mathcal{O}(\tau)$ or $\mathcal{O}(\tau^2)$). An analysis exhibiting this dependency is non-robust, invalidating the claimed convergence rates. Therefore, a core analytical challenge lies in developing mathematical techniques  to control the pressure and divergence errors uniformly with respect to $\varepsilon$.

Notably, the critical issue of $\varepsilon$-uniform error estimation is overlooked in several recent studies of ACM-type methods \cite{prohl1997projection,li2025low,cappanera2025artificial,aggul2023artificial}. A typical example is the work by Prohl \cite{prohl1997projection}, where the approximation error inherently depends on $\varepsilon$, scaling as $\mathcal{O}(\varepsilon(1+\log(1/\varepsilon)))$ (cf. \cite[Theorem 4.1]{prohl1997projection}). In contrast, the dependency on $\varepsilon$ was carefully addressed in the foundational work by Shen \cite{shen1995error,shen1996new} for related penalty methods. However, Shen's analyses were strictly confined to temporal semi-discretization, leaving the fully discrete error estimates unexplored. Furthermore, these results did not account for the influence of the Reynolds number, thus lacking $\nu$-robustness. Extending the semi-discrete analysis to the fully discrete setting is highly non-trivial, as spatial discretization errors (particularly from the nonlinear term) inevitably couple with the parameters $\nu$ and $\varepsilon$, making the derivation of parameter-uniform bounds significantly more challenging. Consequently, there remains a notable gap in the literature concerning fully discrete, high-order ACM schemes with rigorous, parameter-uniform error analyses valid across all Reynolds numbers.

In this work, we propose a grad-div stabilized Crank-Nicolson fully discrete scheme for ACM and rigorously establish its parameter-robust optimal error estimates. The main contributions are summarized as follows:
\begin{itemize}
	\item 
Fully discrete parameter-robust analysis: To the best of our knowledge, this is the first work to provide fully discrete optimal error estimates for a second-order ACM scheme that are simultaneously robust with respect to both the viscosity $\nu$ (Reynolds number) and the artificial compressibility parameter $\varepsilon$. Specifically, we establish optimal convergence rates for the velocity in both the $L^2$ and $H^1$ norms, and for the pressure in the $L^2$ norm. Our analysis avoids the $\varepsilon^{-1}$ dependency that causes order reduction and the $\nu^{-1}$ dependency that limits long-time accuracy.
\item 
Dual grad-div stabilization: We incorporate two grad-div stabilization terms serving distinct purposes. The first ameliorates the inherent violation of incompressibility caused by the continuous perturbation, effectively decoupling the velocity and pressure solves. The second, critically, provides the necessary stability mechanism to ensure Reynolds number robustness.
\item 
Scott-Vogelius spatial discretization: We employ the SV finite element pair on Alfeld refined meshes. This choice ensures that the discrete velocity divergence belongs to the discrete pressure space, which is fundamental for the decoupling strategy and the subsequent $\varepsilon$-robust error analysis.
\end{itemize}

The rest of paper is organized as follows. In Section \ref{prelim}, we will introduce the spatial discretization with the Scott-Vogelius element and some preliminaries for error analysis. Following that, the  second order Crank-Nicolson scheme with grad-div stabilized finite element scheme together with  the energy stability is  presented in Section \ref{sect:acmcn}. Robust $H^1$ error estimate of the proposed  CN-ACM scheme is derived  in Section \ref{sect:rey-robust}. The optimal $L^2$ error estimate is present in Section \ref{sect:opl2} by determining the difference between CN-ACM with the classical Navier-Stokes finite element scheme. Several numerical experiments are carried out in Section \ref{NumExp} to verify the theoretical results and demonstrate the effectiveness and efficiency of the proposed schemes, especially  for problems with large Reynolds numbers. Finally, some concluding remarks are drawn in Section \ref{conclusion}.

\section{Preliminaries and notations}\label{prelim}

Let $\Omega \subset \mathbb{R}^d$ ($d \geq 2$) be a bounded Lipschitz domain. Throughout this paper, we adopt standard notations for Lebesgue and Sobolev spaces. For $1 \leq p \leq \infty$, we denote by $L^p(\Omega)$ the Lebesgue space equipped with the norm $\|\cdot\|_{L^p}$; in the case $p = 2$, the space is endowed with the canonical inner product $(\cdot,\,\cdot)$, and the corresponding norm is simply written as $\|\cdot\|$. For $m \in \mathbb{N}$, $W^{m,p}(\Omega)$ denotes the standard Sobolev space with norm $\|\cdot\|_{W^{m,p}}$, and we write $H^m(\Omega) := W^{m,2}(\Omega)$ with the associated norm $\|\cdot\|_m$ and semi-norm $|\cdot|_m$. The closed subspace $H_0^1(\Omega) \subset H^1(\Omega)$ consists of functions with vanishing trace on $\partial\Omega$, and its dual space is denoted by $H^{-1}(\Omega)$, equipped with the norm $\|\cdot\|_{-1}$.

For time-dependent function spaces, we introduce $L^q(0,T; W^{m,p}(\Omega))$ ($1 \leq q \leq \infty$) as the space of $W^{m,p}(\Omega)$-valued functions equipped with the norm
$$
\|u\|_{L^q(0,T;\,W^{m,p}(\Omega))} :=
\begin{cases}
	\textstyle \displaystyle\left(\int_0^T \|u(t)\|_{W^{m,p}}^q \, dt\right)^{1/q}, & 1 \leq q < \infty, \\[6pt]
	\textstyle \displaystyle\operatorname*{ess\,sup}_{t \in (0,T)} \|u(t)\|_{W^{m,p}}, & q = \infty.
\end{cases}
$$
For brevity, $L^q(0,T; W^{m,p}(\Omega))$ will be abbreviated as $L^q(0,T; W^{m,p})$. Throughout this work, unless stated otherwise, we impose homogeneous Dirichlet boundary conditions, i.e.,
$\bm{u}_D(\bm{x}, t) = \bm{0}$ for all $(\bm{x}, t) \in \partial\Omega \times (0, T]$.
We often suppress explicit dependence on the domain $\Omega$ and the time interval $(0,T)$ when clear from context.

We define the Sobolev space pair $\bm{X} \times M$ as
\begin{equation*}
	\begin{aligned}
		\bm{X} &= \bigl(H_0^1(\Omega)\bigr)^d := \bigl\{ \bv \in \bigl(H^1(\Omega)\bigr)^d \;\big|\; \bv|_{\partial\Omega} = \bm{0} \bigr\}, \\
		M &= L_0^2(\Omega) := \Bigl\{ q \in L^2(\Omega) \;\Big|\; \int_{\Omega} q \, dx = 0 \Bigr\}.
	\end{aligned}
\end{equation*}

By the Sobolev embedding theorems \cite{adams2003sobolev}, for any $\bw \in \bm{X}$ the following estimates hold:
\begin{equation}\label{embed:sob}
	\left\{
	\begin{aligned}
		\|\bw\|_{L^4} &\leq C_{\mathrm{em}}^{(1)} \|\bw\|_{1}, &\quad& d = 2, \\
		\|\bw\|_{L^3} &\leq C_{\mathrm{em}}^{(2)} \|\bw\|_{1}, &\quad& d = 3, \\
		\|\bw\|_{L^6} &\leq C_{\mathrm{em}}^{(3)} \|\bw\|_{1}, &\quad& d = 3,
	\end{aligned}
	\right.
\end{equation}
where the positive constants $C_{\mathrm{em}}^{(i)}$ depend only on $\Omega$. For simplicity, we set $C_{\mathrm{em}} = \max\{C_{\mathrm{em}}^{(1)},\, C_{\mathrm{em}}^{(2)},\, C_{\mathrm{em}}^{(3)}\}$.

\begin{lem}[Poincar\'{e}'s inequality {\cite{adams2003sobolev,evans2022partial}}]
	Let $\Omega$ be a bounded, open subset of $\mathbb{R}^d$. For any $u \in W_0^{1,p}(\Omega)$ with $1 \leq p \leq \infty$, there exists a positive constant $C_{\mathrm{Pi}}$, depending on $p$, $d$, and $\Omega$, such that
	\begin{equation*}
		\|u\|_{L^p(\Omega)} \leq C_{\mathrm{Pi}} \|\nabla u\|_{L^p}.
	\end{equation*}
\end{lem}

%To handle the norm $\|p\|_{-1}$, we require the following Poincar\'{e}-type inequality for functions in $L^2(\Omega)$.
%
%\begin{lem}[{\cite[Proposition~IV.1.7]{boyer2012mathematical}}]\label{h-1:poin}
%	Let $\Omega$ be a connected, bounded, Lipschitz domain in $\mathbb{R}^d$. There exists a constant $C_{\mathrm{Pi},2} > 0$ such that for all $p \in L^2(\Omega)$,
%	\begin{equation*}
%		\|p\|_{-1} \leq C_{\mathrm{Pi},2} \left( \frac{1}{|\Omega|} \left| \int_{\Omega} p \, dx \right| + \|\nabla p\|_{-1} \right).
%	\end{equation*}
%\end{lem}

%We define $C_{\mathrm{Pi}} = \max\{C_{\mathrm{Pi},1},\, C_{\mathrm{Pi},2}\}$. By virtue of Lemma~\ref{h-1:poin}, for any $p \in M$ we can estimate $\|p\|_{-1}$ by $\|\nabla p\|_{-1}$. Moreover, $\|\nabla %p\|_{-1}$ is defined as
%\begin{equation*}
%	\|\nabla p\|_{-1} = \sup_{\bv \in \bm{X},\; \bv \neq \bm{0}} \frac{\displaystyle\int_{\Omega} \nabla p \cdot \bv \, dx}{\|\nabla \bv\|}
%	= -\sup_{\bv \in \bm{X},\; \bv \neq \bm{0}} \frac{\displaystyle\int_{\Omega} p\, \nabla \cdot \bv \, dx}{\|\nabla \bv\|}.
%\end{equation*}

The standard weak formulation of the Navier--Stokes equations~\eqref{NS:orig} reads: find $(\bu(t), p(t)) \in \bm{X} \times M$ for $t \in (0,T]$ such that
\begin{equation}\label{wkf:orig}
	\left\{
	\begin{aligned}
		(\bu_t, \bv) + ((\bu \cdot \nabla)\bu, \bv) + \nu(\nabla\bu, \nabla\bv) - (p, \nabla \cdot \bv) &= (\bff, \bv), &\quad& \forall\, \bv \in \bm{X}, \\
		(\nabla \cdot \bu, q) &= 0, &\quad& \forall\, q \in M.
	\end{aligned}
	\right.
\end{equation}
Since the smoothed de~Rham complex
$$
\mathbb{R} \xrightarrow{\subset} H^2(\Omega) \xrightarrow{\mathrm{curl}} \bm{H}^1(\Omega) \xrightarrow{\mathrm{div}} L^2(\Omega) \longrightarrow 0
$$
is exact on a simply connected domain, formulation~\eqref{wkf:orig} is equivalent to the following grad-div stabilized form:
\begin{equation}\label{wkf:orig:stab}
	\left\{
	\begin{aligned}
		(\bu_t, \bv) + ((\bu \cdot \nabla)\bu, \bv) + \nu(\nabla\bu, \nabla\bv) + \mu(\nabla \cdot \bu, \nabla \cdot \bv) - (p, \nabla \cdot \bv) &= (\bff, \bv), &\quad& \forall\, \bv \in \bm{X}, \\
		(\nabla \cdot \bu, q) &= 0, &\quad& \forall\, q \in M,
	\end{aligned}
	\right.
\end{equation}
where $\mu \geq 0$ is an artificial stabilization parameter.

Let $\mathcal{T}_h$ be a shape-regular and quasi-uniform triangulation of $\Omega$ with mesh size $h = \max\{\mathrm{diam}(K) \mid K \in \mathcal{T}_h\}$. The Scott--Vogelius mixed finite element pair $\bm{X}_h \times M_h$ \cite{fabien2022low} for spatial discretization is defined by
\begin{equation*}
	\bm{X}_h := [\mathbb{P}_k]^d \cap \bm{X}, \qquad
	M_h := \mathbb{P}_{k-1}^{\mathrm{disc}} \cap M,
\end{equation*}
with
$$
\begin{aligned}
	\mathbb{P}_k &= \bigl\{ v_h \in C(\bar{\Omega}) : v_h|_K \in P^k(K),\; \forall\, K \in \mathcal{T}_h \bigr\}, \\
	\mathbb{P}_{k-1}^{\mathrm{disc}} &= \bigl\{ q_h \in L^2(\Omega) : q_h|_K \in P^{k-1}(K),\; \forall\, K \in \mathcal{T}_h \bigr\},
\end{aligned}
$$
where $P^k(K)$ denotes the space of polynomials of degree at most $k$ on the element $K$.
The Scott--Vogelius element on Alfeld splits (i.e., barycenter-refined meshes) satisfies the discrete inf-sup condition \cite{guzman2018inf} when $k \geq d$. Specifically, there exists a constant $\gamma > 0$ such that
\begin{equation}\label{dis:infsup}
	\inf_{q_h \in M_h} \sup_{\bv_h \in \bm{X}_h} \frac{(q_h, \nabla \cdot \bv_h)}{\|q_h\| \, \|\bv_h\|_1} \geq \gamma.
\end{equation}
Consequently, for any $p_h \in M_h$, there exists a unique $\bv_h \in \bm{X}_h$ satisfying
\begin{equation*}
	(\nabla \cdot \bv_h, q_h) = (p_h, q_h), \quad \forall\, q_h \in M_h,
\end{equation*}
with the bound
$$
\|\bv_h\|_1 \leq C_{\mathrm{inf}} \|p_h\|,
$$
where $C_{\mathrm{inf}} > 0$ is independent of $p_h$. In this paper, we focus primarily on the case $d = 2$ and accordingly consider $k = 2$.
In addition, the following approximation properties hold \cite{BrennerScott}:
\begin{equation}\label{ApproxProperties}
	\left\{
	\begin{aligned}
		\inf_{\bv \in \bm{X}_h} \|\bw - \bv\|_k &\leq C_0^{(1)}\, h^{3-k} \|\bw\|_3, &\quad& \forall\, \bw \in (H^3(\Omega))^d,\; k = 0, 1, \\
		\inf_{q \in M_h} \|r - q\| &\leq C_0^{(2)}\, h^2 \|r\|_2, &\quad& \forall\, r \in H^2(\Omega),
	\end{aligned}
	\right.
\end{equation}
where $C_0^{(1)} > 0$ and $C_0^{(2)} > 0$ are constants depending only on $\Omega$. For simplicity, we set $C_0 = \max\{C_0^{(1)},\, C_0^{(2)}\}$.

We define $\bm{V}$ as the divergence-free subspace of $\bm{X}$ and $\bm{V}_h$ as the divergence-free subspace of $\bm{X}_h$:
\begin{equation*}
	\bm{V} := \bigl\{ \bv \in \bm{X} \;\big|\; \nabla \cdot \bv = 0 \text{ a.e.\ in } \Omega \bigr\}, \qquad
	\bm{V}_h := \bigl\{ \bv_h \in \bm{X}_h \;\big|\; (q_h, \nabla \cdot \bv_h) = 0,\; \forall\, q_h \in M_h \bigr\}.
\end{equation*}
We note that $\bm{V}_h \subseteq \bm{V}$ in the case of Scott--Vogelius elements.

Next, we introduce the Stokes projection $(\mathrm{P}_{\mathrm{st}},\, \mathrm{Q}_{\mathrm{st}})$: for any given $(\bw, q) \in \bm{X} \times M$, find $(\mathrm{P}_{\mathrm{st}}\bw,\, \mathrm{Q}_{\mathrm{st}}q) \in \bm{X}_h \times M_h$ such that
\begin{equation}\label{stokes:proj}
	\nu(\nabla(\bw - \mathrm{P}_{\mathrm{st}}\bw), \nabla\bv_h) - (\nabla \cdot \bv_h,\, q - \mathrm{Q}_{\mathrm{st}}q) + (\nabla \cdot (\bw - \mathrm{P}_{\mathrm{st}}\bw),\, q_h) = 0, \quad \forall\, (\bv_h, q_h) \in \bm{X}_h \times M_h.
\end{equation}
The stability and approximation properties of the Stokes projection~\eqref{stokes:proj} are summarized in the following lemma.

\begin{lem}[Stability and approximation of the Stokes projection {\cite{GiraultNochettoScott,chen2006pointwise,garcia2021convergence,SchroederLube}}]
	\label{stokes_stability}
	For any fixed $r \in [2, \infty]$, there exists a constant $C_{\mathrm{st}} > 0$ depending only on $\Omega$ such that
	\begin{equation*}
		\|\nabla \mathrm{P}_{\mathrm{st}} \bw\|_{L^r} \leq C_{\mathrm{st}} \|\nabla \bw\|_{L^r}, \quad \forall\, \bw \in \bm{V}.
	\end{equation*}
	In addition, the following error estimates hold:
	\begin{equation*}
		\left\{
		\begin{aligned}
			\|\bw - \mathrm{P}_{\mathrm{st}} \bw\| + h\, \|\bw - \mathrm{P}_{\mathrm{st}} \bw\|_1 &\leq C_{\mathrm{st}}\, h^3 \|\bw\|_3, \\
			\|q - \mathrm{Q}_{\mathrm{st}} q\| &\leq C_{\mathrm{st}}\, h^2 \bigl(\|\bw\|_3 + \|q\|_2\bigr).
		\end{aligned}
		\right.
	\end{equation*}
\end{lem}
We also recall the following standard inverse inequalities \cite{BrennerScott}.
\begin{lem}[Inverse inequalities]
	For any fixed $0 \leq n \leq m \leq 1$ and $1 \leq q \leq p \leq \infty$, there exists a constant $\overline{C}_{\mathrm{inv}} > 0$, depending only on $\Omega$, such that
	\begin{equation*}
		\|\bw_h\|_{W^{m,p}} \leq \overline{C}_{\mathrm{inv}}\, h^{n - m - d\left(\frac{1}{q} - \frac{1}{p}\right)} \|\bw_h\|_{W^{n,q}}, \quad \forall\, \bw_h \in \bm{X}_h.
	\end{equation*}
\end{lem}

\noindent In particular, we will use the following specific inverse estimates:
\begin{equation}\label{666}
	\|\bw_h\|_{L^{\infty}} \leq \overline{C}_{\mathrm{inv}}^{(1)}\, h^{-d/2} \|\bw_h\|, \qquad
	\|\bw_h\|_{H^1} \leq \overline{C}_{\mathrm{inv}}^{(2)}\, h^{-1} \|\bw_h\|, \qquad
	\|\bw_h\|_{L^{2d/(d-1)}} \leq \overline{C}_{\mathrm{inv}}^{(3)}\, h^{-1/2} \|\bw_h\|.
\end{equation}
We define $C_{\mathrm{inv}} = \max\bigl\{\overline{C}_{\mathrm{inv}}^{(1)},\, \overline{C}_{\mathrm{inv}}^{(2)},\, \overline{C}_{\mathrm{inv}}^{(3)}\bigr\}$ for simplicity.

For the temporal discretization, we consider a uniform partition of the time interval $[0,T]$ with time step size $\tau = T/N$ and denote $t_n = n\tau$ for $n = 0, 1, \ldots, N$. For a generic variable $\varphi$, we let $\varphi^n$ denote its value at $t = t_n$, i.e., $\varphi^n = \varphi(t_n)$.
The following consistency error estimates will be used frequently for temporal error analysis; see the appendix of~\cite{ErvinLaytonNeda}.

\begin{lem}[Temporal consistency errors {\cite{ErvinLaytonNeda}}]
	\label{lemp1}
	Assume $\bu \in C^0(t_n, t_{n+1}; L^2(\Omega))$ and $\bu_{tt} \in L^2((t_n, t_{n+1}) \times \Omega)$. Then
	\begin{equation}\label{ConsistencyError1}
	\textstyle	\left\| \frac{\bu(t_{n+1}) + \bu(t_n)}{2} - \bu(t_{n+1/2}) \right\|^2 \leq \frac{\tau^3}{48} \int_{t_n}^{t_{n+1}} \|\bu_{tt}\|^2 \, dt.
	\end{equation}
	If $\nabla\bu \in C^0(t_n, t_{n+1}; L^2(\Omega))$ and $\nabla\bu_{tt} \in L^2((t_n, t_{n+1}) \times \Omega)$, then
	\begin{equation}\label{ConsistencyError2}
	\textstyle	\left\| \nabla\!\left( \frac{\bu(t_{n+1}) + \bu(t_n)}{2} - \bu(t_{n+1/2}) \right) \right\|^2 \leq \frac{\tau^3}{48} \int_{t_n}^{t_{n+1}} \|\nabla\bu_{tt}\|^2 \, dt.
	\end{equation}
	If $\bu_t \in C^0(t_n, t_{n+1}; L^2(\Omega))$ and $\bu_{ttt} \in L^2((t_n, t_{n+1}) \times \Omega)$, then
	\begin{equation}\label{ConsistencyError3}
	\textstyle	\left\| \frac{\bu(t_{n+1}) - \bu(t_n)}{\tau} - \bu_t(t_{n+1/2}) \right\|^2 \leq \frac{\tau^3}{1280} \int_{t_n}^{t_{n+1}} \|\bu_{ttt}\|^2 \, dt.
	\end{equation}
\end{lem}

Finally, we state the discrete Gr\"onwall inequality that will be employed in the error analysis.

\begin{lem}[Discrete Gr\"onwall lemma {\cite{john2016finite}}]\label{gronwall1}
	Let $\tau$, $H$, and $a_n, b_n, c_n, d_n$ (for integers $n \geq 0$) be finite non-negative numbers satisfying
	\begin{equation*}
	\textstyle	a_l + \tau \sum\limits_{n=0}^{l} b_n \leq \tau \sum\limits_{n=0}^{l-1} d_n\, a_n + \tau \sum\limits_{n=0}^{l} c_n + H, \quad \forall\, l \geq 1.
	\end{equation*}
	Then
	\begin{equation*}
	\textstyle	a_l + \tau \sum\limits_{n=0}^{l} b_n \leq \exp\!\left( \tau \sum\limits_{n=0}^{l-1} d_n \right) \left( \tau \sum\limits_{n=0}^{l} c_n + H \right), \quad \forall\, l \geq 1.
	\end{equation*}
\end{lem}

	\section{Second-order Crank-Nicolson ACM scheme  with grad-div stabilization}\label{sect:acmcn}

	%Introduce $\varepsilon$ as a tuning parameter and $\mu$ as the stabilization parameter. 
	Define $\overline{\phi}^{n-1/2}:=\frac{\phi^n+\phi^{n-1}}{2}$. We propose the following second-order Crank-Nicolson (ACM-CN) scheme as: given  $(\bu_h^{n-1}, p_h^{n-1})$, find $(\bu_h^n, p_h^n)\in \bX_h\times M_h$, such that
			\begin{subequations}\label{acm:gdiv:eq}
		\begin{empheq}[left=\empheqlbrace]{align}
			&\textstyle
			\left(\frac{\bu_h^n-\bu_h^{n-1}}{\tau}, \bv_h\right)+\nu\left(\nabla \overline{\bu}_h^{n-1/2}, \nabla \bv_h\right) +\mu\left(\nabla\cdot\overline{\bu}_h^{n-1/2}, \nabla\cdot\bv_h\right) +c\left(\overline{\bu}_h^{n-1 / 2}, \overline{\bu}_h^{n-1/2}, \bv_h\right)\notag \\
			& \qquad\qquad\qquad\qquad\qquad\qquad\qquad\qquad
			\textstyle
			-\left(  \overline{p}_h^{n-1/2}, \nabla \cdot \bv_h\right)=\left(f^{n-1 / 2}, \bv_h\right),  \forall\ \bv_h\in \bX_h,\label{CN:mom:eq}\\
			& \textstyle
			\varepsilon \left( \frac{p_h^n-p_h^{n-1}}{\tau}, q_h\right) + \left( \nabla\cdot\overline{\bu}_h^{n-1 / 2}, q_h\right)=0, \forall \ q_h\in M_h,\label{CN:press:eq}
		\end{empheq}
	\end{subequations}
	where the nonlinear term $\textstyle c\left(\overline{\bu}_h^{n-1 / 2}, \overline{\bu}_h^{n-1/2}, \bv_h\right)=\frac12 \left((\overline{\bu}_h^{n-1 / 2}\cdot\nabla) \overline{\bu}_h^{n-1/2}, \bv_h\right) - \frac12\left( (\overline{\bu}_h^{n-1 / 2}\cdot\nabla)\bv_h,  \overline{\bu}_h^{n-1/2}\right)$ is the skew-symmetric form.  We can see that \eqref{acm:gdiv:eq} is the velocity-pressure coupled system. To boost the benifit of ACM method, we can decouple the velocity and pressure by the following transformation. First of all, \eqref{CN:press:eq} can be reformulated as
	\begin{equation*}
	\textstyle 	\left(  \overline{p}_h^{n-1/2}, q_h\right) = \left( p_h^{n-1}, q_h\right) - \frac{\tau}{2\varepsilon}\left( \nabla\cdot \overline{\bu}_h^{n-1 / 2}, q_h\right).
	\end{equation*}
Since we employ the SV element is chosen, the above equation can be reformulated as
	\begin{equation*}
\textstyle	\left(  \overline{p}_h^{n-1/2}, \nabla\cdot\bv_h\right) = \left( p_h^{n-1}, \nabla\cdot\bv_h\right) - \frac{\tau}{2\varepsilon}\left( \nabla\cdot \overline{\bu}_h^{n-1 / 2}, \nabla\cdot\bv_h\right).
\end{equation*}
Substituting the above equation into \eqref{CN:mom:eq}, CN-ACM scheme \eqref{acm:gdiv:eq} is changed to
	\begin{subequations}\label{acm:gdiv}
	\begin{empheq}[left=\empheqlbrace]{align}
		&\textstyle
		\left(\frac{\bu_h^n-\bu_h^{n-1}}{\tau}, \bv_h\right)+\nu\left(\nabla \overline{\bu}_h^{n-1/2}, \nabla \bv_h\right) +\mu\left(\nabla\cdot\overline{\bu}_h^{n-1/2}, \nabla\cdot\bv_h\right)+c\left(\overline{\bu}_h^{n-1 / 2}, \overline{\bu}_h^{n-1/2}, \bv_h\right)\notag \\
		& \qquad\qquad\qquad
		\textstyle
		+  \frac{\tau}{2\varepsilon}  \left(\nabla \cdot \overline{\bu}_h^{n-1/2}, \nabla \cdot \bv_h\right)
		-\left( p_h^{n-1}, \nabla \cdot \bv_h\right)=\left(f^{n-1 / 2}, \bv_h\right),  \forall\ \bv_h\in \bX_h,\label{CN:mom}\\
		& \textstyle
		\varepsilon \left( \frac{p_h^n-p_h^{n-1}}{\tau}, q_h\right) + \left( \nabla\cdot\overline{\bu}_h^{n-1 / 2}, q_h\right)=0, \forall \ q_h\in M_h.\label{CN:press}
	\end{empheq}
\end{subequations}
In \eqref{acm:gdiv}, we can see that the velocity and pressure are decoupled.  In practice, scheme \eqref{acm:gdiv} is implemented to achieve velocity-pressure decoupling, while the equivalent formulation \eqref{acm:gdiv:eq} is utilized for the theoretical error analysis. We first establish the energy stability of the CN-ACM scheme  \eqref{acm:gdiv:eq}.
\begin{lem}
	Given $\tau\leq 1$, the CN-ACM scheme \eqref{acm:gdiv:eq} satisfies the following estimate
	\begin{equation*}
	\Vert \bu_h^n\Vert^2   + \varepsilon  \Vert p_h^n \Vert^2   +2\tau \nu \sum\limits_{i=1}^n \Vert \nabla \overline{\bu}_h^{i-1/2}\Vert^2 +2\tau\mu\sum\limits_{i=1}^n \Vert \nabla\cdot\overline{\bu}_h^{i-1/2}\Vert^2 
	\textstyle \leq  2 \exp(T)\left(\Vert \bu_h^{0}\Vert^2 + \varepsilon \Vert p_h^{0}\Vert^2 + 2\tau  \sum\limits_{i=1}^n\Vert f^{i-1 / 2} \Vert^2\right).
	\end{equation*} 
\end{lem}
\begin{proof}
	Let $\bv_h=\overline{\bu}_h^{n-1/2}$ in \eqref{CN:mom:eq} and $q_h=\overline{p}_h^{n-1/2}$ in \eqref{CN:press:eq}, we have
	\begin{equation*}
		\textstyle
		 \frac{\Vert \bu_h^n\Vert^2 - \Vert \bu_h^{n-1}\Vert^2}{2\tau} + \varepsilon  \frac{\Vert p_h^n \Vert^2 - \Vert p_h^{n-1}\Vert^2}{2\tau} +\nu \Vert \nabla \overline{\bu}_h^{n-1/2}\Vert^2 +\mu\Vert \nabla\cdot\overline{\bu}_h^{n-1/2}\Vert^2=\left(f^{n-1 / 2}, \overline{\bu}_h^{n-1/2}\right)\leq \Vert f^{n-1 / 2} \Vert^2 + \frac14 \Vert \overline{\bu}_h^{n-1/2}\Vert^2.
	\end{equation*}
	Changing the index $n$ to $i$, summing over $i$ from $1$ to $n$ and multiplying $2\tau$ on both sides, it yields
	 \begin{equation*}
	 	\begin{aligned}
	 	\textstyle
	 	 \Vert \bu_h^n\Vert^2 - \Vert \bu_h^{0}\Vert^2 + \varepsilon  \Vert p_h^n \Vert^2 - \varepsilon \Vert p_h^{0}\Vert^2  +2\tau \nu \sum\limits_{i=1}^n \Vert \nabla \overline{\bu}_h^{i-1/2}\Vert^2 +2\tau\mu\sum\limits_{i=1}^n \Vert \nabla\cdot\overline{\bu}_h^{i-1/2}\Vert^2\\
	 	 \textstyle \leq 2\tau  \sum\limits_{i=1}^n\Vert f^{i-1 / 2} \Vert^2 + \frac12 \sum\limits_{i=1}^n \tau \Vert \bu_h^i\Vert^2.
   \end{aligned}
	 \end{equation*}
	 By Gr\"onwall's inequality, we can have
	 	 \begin{equation*}
	 	\begin{aligned}
	 		\textstyle
	 		\Vert \bu_h^n\Vert^2   + \varepsilon  \Vert p_h^n \Vert^2   +2\tau \nu \sum\limits_{i=1}^n \Vert \nabla \overline{\bu}_h^{i-1/2}\Vert^2 +2\tau\mu\sum\limits_{i=1}^n \Vert \nabla\cdot\overline{\bu}_h^{i-1/2}\Vert^2 
	 		\textstyle \leq  2 \exp(T)\left(\Vert \bu_h^{0}\Vert^2 + \varepsilon \Vert p_h^{0}\Vert^2 + 2\tau  \sum\limits_{i=1}^n\Vert f^{i-1 / 2} \Vert^2\right).
	 	\end{aligned}
	 \end{equation*}
\end{proof}
\begin{remark}
	Alternatively, one could consider the following linearized CN-ACM scheme. For $n\geq 2$,  given $(\bu_h^{n-1}, p_h^{n-1})$, find $(\bu_h^n, p_h^n)\in \bX_h\times M_h$, such that 
		\begin{subequations}\label{acm:gdiv:lin}
		\begin{empheq}[left=\empheqlbrace]{align}
			&\textstyle
			\left(\frac{\bu_h^n-\bu_h^{n-1}}{\tau}, \bv_h\right)+\nu\left(\nabla \overline{\bu}_h^{n-1/2}, \nabla \bv_h\right) +\mu\left(\nabla\cdot\overline{\bu}_h^{n-1/2}, \nabla\cdot\bv_h\right)+c\left(\widetilde{\bu}_h^{n-1 / 2}, \overline{\bu}_h^{n-1/2}, \bv_h\right)\notag \\
			& \qquad\qquad\qquad
			\textstyle
			+  \frac{\tau}{2\varepsilon}  \left(\nabla \cdot \overline{\bu}_h^{n-1/2}, \nabla \cdot \bv_h\right)
			-\left( p_h^{n-1}, \nabla \cdot \bv_h\right)=\left(f^{n-1 / 2}, \bv_h\right),  \forall\ \bv_h\in \bX_h,\label{CN:mom:lin}\\
			& \textstyle
			\varepsilon \left( \frac{p_h^n-p_h^{n-1}}{\tau}, q_h\right) + \left( \nabla\cdot\overline{\bu}_h^{n-1 / 2}, q_h\right)=0, \forall \ q_h\in M_h,\label{CN:press:lin}
		\end{empheq}
	\end{subequations}
	where $ \widetilde{\bu}_h^{n-1 / 2} = \frac32 \bu_h^{n-1}-\frac12\bu_h^{n-2}$ is the extrapolation at $t=t_{n-1/2}$. As for the first step, we can consider the following backward Euler (BE-ACM) scheme: given $(\bu_h^0, p_h^0)$, find $(\bu_h^1, p_h^1)\in \bX_h\times M_h$, such that 
	\begin{subequations}\label{acm:gdiv:lin:be}
		\begin{empheq}[left=\empheqlbrace]{align}
			&\textstyle
			\left(\frac{\bu_h^1-\bu_h^{0}}{\tau}, \bv_h\right)+\nu\left(\nabla \bu_h^{1}, \nabla \bv_h\right) +\mu\left(\nabla\cdot{\bu}_h^{1}, \nabla\cdot\bv_h\right)+c\left({\bu}_h^{0}, {\bu}_h^{1}, \bv_h\right)\notag \\
			& \qquad\qquad\qquad
			\textstyle
			+  \frac{\tau}{2\varepsilon}  \left(\nabla \cdot {\bu}_h^{1}, \nabla \cdot \bv_h\right)
			-\left( p_h^{0}, \nabla \cdot \bv_h\right)=\left(f^{1}, \bv_h\right),  \forall\ \bv_h\in \bX_h,\label{BE:mom:lin}\\
			& \textstyle
			\varepsilon \left( \frac{p_h^1-p_h^{0}}{\tau}, q_h\right) + \left( \nabla\cdot{\bu}_h^{1}, q_h\right)=0, \forall \ q_h\in M_h.\label{BE:press:lin}
		\end{empheq}
	\end{subequations}
	The analysis of such a linearized scheme necessitates separate estimates for both the CN-ACM and BE-ACM formulations. To maintain simplicity, we restrict our attention to the CN-ACM schemes \eqref{acm:gdiv:eq} and \eqref{acm:gdiv}. Furthermore, we include a remark to highlight the differences in estimating the nonlinear terms when linearization is employed.
\end{remark}

	\section{Reynolds number robust error estimate on CN-ACM scheme}\label{sect:rey-robust}
In order to show the error estimate on the velocity and pressure, we need the following regularities on the real solution $(\bu(t), p(t))$ of \eqref{NS:orig}. That is
	\begin{equation}\label{2nd:regularity}
	\left\{
	\begin{aligned}
		&\bu\in L^{\infty}(0,T;(H^1_0(\Omega))^d\cap (H^3(\Omega))^d),\bu_t\in L^{\infty}(0,T;(H^3(\Omega))^d),\\
		&\;\bu_{tt}\in L^{\infty}(0,T;(L^2(\Omega))^d), \;\bu_{ttt} \in L^2(0,T;(H^1(\Omega))^d),\\
		&p\in L^{\infty}(0,T;H^2(\Omega)). 
	\end{aligned}
	\right.
\end{equation}
Define $e_u^n:= {\rm P}_{\text{st}}\bu(t_n) - \bu_h^n$ and $e_p^n:={\rm Q}_{\text{st}}p(t_n) - p_h^n$. Introduce $\mathcal{C}(t_n)$ as
	\begin{equation*}
		\begin{aligned}
			\mathcal{C}(t_n):= &~
				\frac{1}{80} \Vert \bu_{ttt}\Vert_{L^2(0, t_n;L^2)}^2 +  16C_{\text{st}}^2   \Vert \bu_t\Vert_{L^2(0, t_n;H^3)}^2\\
			&\textstyle ~  + \frac{ C_{\text{st}}^2}{3} \Vert \bu_{tt}\Vert_{L^2(0, t_n;H^1)}^2  + \frac{ C_{\text{st}}^2  }{3\mu } \Vert p_{tt}\Vert_{L^2(0, t_n;L^2)}^2
			+\frac{2 }{3}C_{\text{em}}^4  \Vert \bu_{tt}\Vert_{L^2(0,t_n;H^1)}^2\Vert \bu \Vert_{L^\infty(0,T;H^2)}^2
			\\
			&\textstyle ~ 
			+  \frac{4 }{3} C_{\text{Pi}}^2 C_{\text{st}}^2 \Vert \nabla \bu\Vert_{L^\infty(0,T;L^\infty)}^2   \Vert \bu_{tt}\Vert_{L^2(0,t_n;H^1)}^2
			+  16 t_nC_{\text{em}}^4 C_{\text{st}}^2  \Vert \bu \Vert_{L^\infty(0,T;H^3)}^2\Vert \bu \Vert_{L^\infty(0,T;H^2)}^2	\\
			&\textstyle~ 
			+ 32 t_n  C_{\text{Pi}}^2 C_{\text{st}}^4 \Vert \nabla \bu \Vert_{L^\infty(0,T;L^\infty)}^2 \Vert \bu \Vert_{L^\infty(0,T;H^3)}^2
			+ 
			\frac{\varepsilon }{160} \Vert p_{ttt}\Vert_{L^2(0, t_n;L^2)}^2 +  8 \varepsilon C_{\text{st}}^2  \Vert p_t\Vert_{L^2(0, t_n;H^2)}^2 			\\
			&\textstyle ~
			+ 4\tau\sum\limits_{i=1}^n \varepsilon C_0^2  \Vert  \frac{\partial p(t_{i-1/2})}{\partial t} \Vert_2^2  
			+  2 \tau\sum\limits_{i=1}^n  \left(4 \left( \nu+\mu\right)+\frac34 \right)  C_{\text{inf}}^2\Vert \frac{\partial p(t_{i-1/2})}{\partial t} \Vert^2 \\
			&\textstyle  ~   +   4\tau\sum\limits_{i=1}^n C_{\text{Pi}}^2C_{\text{inf}}^2 \left(\frac12 +   C_{\text{st}} \Vert \nabla \bu \Vert_{L^\infty(0,T;L^\infty)} +\frac{1}{2\mu} C_{\text{Pi}}^2 C_{\text{st}}^2\Vert \bu \Vert_{L^\infty(0,T;L^\infty)}  \right)  \Vert  \frac{\partial p(t_{i-1/2})}{\partial t} \Vert^2 .
		\end{aligned}
	\end{equation*}
	Then, we have the following Reynolds number and perturbation parameter $\varepsilon$ robust estimate.
	\begin{thm}\label{op:cn:h1} Given $(\bu(t), p(t_n))$ is the solution to \eqref{NS:orig} satisfying the regulrities \eqref{2nd:regularity} and $(\bu_h^n, p_h^n)$ is the solution to CN-ACM scheme \eqref{acm:gdiv:eq}. There exists $\tau_1>0$, when $\tau\leq \min\{\tau_1, 1\}$ and $h\leq 1$, then following estimate holds
		\begin{equation*}
			\begin{aligned}
				&\textstyle	 \Vert e_u^n \Vert^2 + \varepsilon   \Vert e_p^n \Vert^2   + \nu \tau\sum\limits_{i=1}^n \Vert \nabla\overline{e}_u^{i-1/2} \Vert^2 + \mu  \tau\sum\limits_{i=1}^n \Vert \nabla\cdot \overline{e}_u^{i-1/2} \Vert^2\\
				&\leq ~ \textstyle
				\exp\left( t_n\left(\frac52+  4 C_{\text{st}} \Vert \nabla \bu \Vert_{L^\infty(0,T;L^\infty)} +\frac{2}{ \mu} C_{\text{Pi}}^2C_{\text{st}}^2\Vert \nabla\bu \Vert_{L^\infty(0,T;L^\infty)} \right) \right) \mathcal{C}(t_n)(\tau^4 + h^4 + \varepsilon^2).
			\end{aligned}
		\end{equation*}
	\end{thm}
	\begin{proof}
	The error equations are given as
	\begin{subequations}\label{acm:gdiv:err:1}
		\begin{empheq}[left=\empheqlbrace]{align}
			& \textstyle
			\left(\frac{e_u^n-e_u^{n-1}}{\tau}, \bv_h\right)+\nu\left(\nabla \overline{e}_u^{n-1/2}, \nabla \bv_h\right)+c\left(\bu\left(t_{n-1/2}\right), \bu\left(t_{n-1/2}\right), \bv_h\right)
			-c\left(\overline{\bu}_h^{n-1 / 2}, \overline{\bu}_h^{n-1/2}, \bv_h\right) \notag \\
			& \quad \textstyle
			+ \mu\left(\nabla \cdot \overline{e}_u^{n-1/2}, \nabla \cdot \bv_h\right)-\left(\overline{e}_p^{n-1/2}, \nabla \cdot \bv_h\right)
			=-\left(\frac{\partial \bu\left(t_{n-1/2} \right)}{\partial t}-{\rm P}_{\text{st}} \frac{\bu\left(t_n\right)-\bu\left(t_{n-1}\right)}{\tau}, \bv_h\right)
			\notag \\
			& \textstyle \qquad\quad
			-\nu\left(\nabla{\rm P}_{\text{st}}\left( \bu\left(t_{n-1/2}\right)-\overline{\bu}\left(t_{n-1/2}\right)\right), \nabla \bv_h\right) 
			+\left( {\rm Q}_{\text{st}}\left( p\left(t_{n-1/2}\right)-\overline{p}\left(t_{n-1/2}\right) \right), \nabla \cdot \bv_h\right), \label{acm:gdiv:err}\\
			& \textstyle
			\varepsilon\left(\frac{e_p^n-e_p^{n-1}}{\tau}, q_h\right)=~\textstyle
			-\varepsilon\left(\frac{\partial p\left(t_{n-1/2}\right)}{\partial t}-{\rm Q}_{\text{st}} \frac{p\left(t_n\right) - p\left(t_{n-1}\right)}{\tau}, q_h\right)
			%-\left(\nabla \cdot\left(\bu\left(t_{n-1/2}\right)-\frac{\bu\left(t_n\right)+\bu\left(t_{n-1}\right)}{2}\right), q_h\right)
			%-\left( \nabla\cdot\left( {\rm P}_{\text{st}}\frac{\bu(t_n)+\bu(t_{n-1})}{2} - \tilde{\bu}_h^{n-1/2}\right), q_h\right)
			-\left(  \nabla\cdot\overline{e}_u^{n-1/2}, q_h\right)
			+ \varepsilon \left(\frac{\partial p(t_{n-1/2})}{\partial t}, q_h\right). \label{acm:gdiv:err:p}
		\end{empheq}
	\end{subequations}

Let $\bv_h = \overline{e}_u^{n-1/2}$ in \eqref{acm:gdiv:err} and $q_h = \overline{e}_p^{n-1/2}$ in \eqref{acm:gdiv:err:p}, then combine them,  we have
\begin{equation}\label{cn:1st}
	\begin{aligned}
		&\textstyle	\frac{\Vert e_u^n \Vert^2 - \Vert e_u^{n-1} \Vert^2}{2\tau} + \varepsilon \frac{ \Vert e_p^n \Vert^2 - \Vert  e_p^{n-1} \Vert^2}{2\tau} +\nu \Vert \nabla\overline{e}_u^{n-1/2} \Vert^2 +\mu \Vert \nabla\cdot \overline{e}_u^{n-1/2} \Vert^2\\
		& \textstyle 
		= ~ \textstyle
		-\left(\frac{\partial \bu\left(t_{n-1/2} \right)}{\partial t}-{\rm P}_{\text{st}} \frac{\bu\left(t_n\right)-\bu\left(t_{n-1}\right)}{\tau}, \overline{e}_u^{n-1/2}\right)
		-\nu\left(\nabla{\rm P}_{\text{st}} \left(\bu\left(t_{n-1/2}\right)-\overline{\bu}\left(t_{n-1/2}\right)\right), \nabla \overline{e}_u^{n-1/2} \right) \\
		&\textstyle \quad
		+\left({\rm Q}_{\text{st}} \left(  p\left(t_{n-1/2}\right)-\overline{p}\left(t_{n-1/2}\right) \right), \nabla \cdot \overline{e}_u^{n-1/2}\right) - c\left(\bu\left(t_{n-1/2}\right), \bu\left(t_{n-1/2}\right), \overline{e}_u^{n-1/2}\right)  
		\\
		&\textstyle \quad
		+ c\left(\overline{\bu}_h^{n-1 / 2}, \overline{\bu}_h^{n-1/2}, \overline{e}_u^{n-1/2}\right) -\varepsilon\left(\frac{\partial p\left(t_{n-1/2}\right)}{\partial t}-{\rm Q}_{\text{st}} \frac{p\left(t_n\right) - p\left(t_{n-1}\right)}{\tau}, \overline{e}_p^{n-1/2}\right)
		+ \varepsilon \left(\frac{\partial p(t_{n-1/2})}{\partial t}, \overline{e}_p^{n-1/2}\right)\\
		&:=\sum\limits_{i=1}^7\mathbb{I}_i.
	\end{aligned}
\end{equation}
\textbf{Part I.} In this part, we bound the first six terms $\mathbb{I}_i$ ($i=1,\dots,6$). While most of these terms can be handled via Young's inequality, the main challenge lies in establishing estimates that are robust with respect to the Reynolds number. To this end, particular care is taken to avoid introducing factors of $1/\nu$ when treating the nonlinear terms.
\begin{equation}\label{mom:vel:time}
	\begin{aligned}
		\textstyle 
		\mathbb{I}_1 &
		\textstyle 
		=~ -\left(\frac{\partial \bu\left(t_{n-1/2} \right)}{\partial t}-\frac{\bu\left(t_n\right)-\bu\left(t_{n-1}\right)}{\tau}, \overline{e}_u^{n-1/2}\right)
		-\left(\frac{\bu\left(t_n\right)-\bu\left(t_{n-1}\right)}{\tau}-{\rm P}_{\text{st}} \frac{\bu\left(t_n\right)-\bu\left(t_{n-1}\right)}{\tau}, \overline{e}_u^{n-1/2}\right)\\
		&\leq ~ \textstyle 
		\frac18\left(\Vert e_u^n\Vert^2 + \Vert  e_u^{n-1}\Vert^2\right) + 
		\frac{\tau^3}{640} \Vert \bu_{ttt}\Vert_{L^2(t_{n-1}, t_n;L^2)}^2 +  \frac{2C_{\text{st}}^2h^6 }{\tau} \Vert \bu_t\Vert_{L^2(t_{n-1}, t_n;H^3)}^2.
	\end{aligned}
\end{equation}
Similarly,
\begin{equation} 
	\begin{aligned}
		\textstyle 
		\mathbb{I}_6  
		&\leq ~ \textstyle 
		\frac{\varepsilon}{8}\left(\Vert e_p^n\Vert^2 + \Vert  e_p^{n-1}\Vert^2\right) + 
		\frac{\varepsilon\tau^3}{640} \Vert p_{ttt}\Vert_{L^2(t_{n-1}, t_n;L^2)}^2 +  \frac{2\varepsilon C_{\text{st}}^2h^4 }{\tau} \Vert p_t\Vert_{L^2(t_{n-1}, t_n;H^2)}^2.
	\end{aligned}
\end{equation}
In addition,
\begin{equation}
	\left\{
	\begin{aligned}
		\textstyle 
		\mathbb{I}_2 \leq & ~\frac{\nu}{8}\Vert \nabla \overline{e}_u^{n-1/2}\Vert^2 + \frac{\nu C_{\text{st}}^2\tau^3}{24} \Vert \bu_{tt}\Vert_{L^2(t_{n-1}, t_n;H^1)}^2,\\
		\mathbb{I}_3 \leq &~\frac{\mu}{8}\Vert \nabla\cdot\overline{e}_u^{n-1/2}\Vert^2 + \frac{ C_{\text{st}}^2\tau^3}{24\mu } \Vert p_{tt}\Vert_{L^2(t_{n-1}, t_n;L^2)}^2.
	\end{aligned}
	\right.   
\end{equation}
Then, for the nonlinear terms, we have
\begin{equation}
	\begin{aligned}
		\mathbb{I}_4 + \mathbb{I}_5 = &~ -c(\bu(t_{n-1/2})-\overline{\bu}(t_{n-1/2}), \bu(t_{n-1/2}), \overline{e}_u^{n-1/2}) -c(\overline{\bu}(t_{n-1/2})-{\rm P}_{\text{st}} \overline{\bu}(t_{n-1/2}), \bu(t_{n-1/2}), \overline{e}_u^{n-1/2})\\
		&~ -c(\overline{e}_u^{n-1/2}, {\rm P}_{\text{st}}\overline{\bu}(t_{n-1/2}), \overline{e}_u^{n-1/2})  - c( {\rm P}_{\text{st}} \overline{\bu}(t_{n-1/2}), \bu(t_{n-1/2}) - \overline{\bu}(t_{n-1/2}), \overline{e}_u^{n-1/2}) \\
		&~ - c({\rm P}_{\text{st}} \overline{\bu}(t_{n-1/2}), \overline{\bu}(t_{n-1/2})-{\rm P}_{\text{st}} \overline{\bu}(t_{n-1/2}), \overline{e}_u^{n-1/2}) 
		  := \sum\limits_{i=1}^5 \mathbb{H}_i,
	\end{aligned}
\end{equation}
where
\begin{equation*}
	\begin{aligned}
		\mathbb{H}_1 \leq & \textstyle ~C_{\text{em}}^2 \Vert \bu(t_{n-1/2})-\overline{\bu}(t_{n-1/2}) \Vert_1  \Vert  \bu(t_{n-1/2}) \Vert_2 \Vert \overline{e}_u^{n-1/2} \Vert\\
		\leq &\textstyle ~ \frac{1}{32}\Vert e_u^n\Vert^2 +  \frac{1}{32}\Vert e_u^{n-1}\Vert^2 + \frac{\tau^3}{12}C_{\text{em}}^4 \Vert \bu_{tt}\Vert_{L^2(t_{n-1},t_n;H^1)}^2\Vert \bu \Vert_{L^\infty(0,T;H^2)}^2,\\
		\mathbb{H}_2 \leq &\textstyle ~C_{\text{em}}^2 \Vert \overline{\bu}(t_{n-1/2})-{\rm P}_{\text{st}} \overline{\bu}(t_{n-1/2}) \Vert_1  \Vert  \bu(t_{n-1/2}) \Vert_2 \Vert \overline{e}_u^{n-1/2} \Vert\\
		\leq&\textstyle ~  \frac{1}{32}\Vert e_u^n\Vert^2 +  \frac{1}{32}\Vert e_u^{n-1}\Vert^2 +  2C_{\text{em}}^4 C_{\text{st}}^2 h^4 \Vert \bu \Vert_{L^\infty(0,T;H^3)}^2\Vert \bu \Vert_{L^\infty(0,T;H^2)}^2.
	\end{aligned}
\end{equation*}
By integration-by-parts, we have
\begin{equation*}
		\mathbb{H}_3 \leq \textstyle ~\left( \frac12 C_{\text{st}}\Vert \nabla \bu \Vert_{L^\infty(0,T;L^\infty)}  + \frac{1}{4\mu}C_{\text{Pi}}^2 C_{\text{st}}^2 \Vert \nabla \bu \Vert_{L^\infty(0,T;L^\infty)}^2 \right) \left( \Vert e_u^n\Vert^2 +  \Vert e_u^{n-1}\Vert^2 \right) + \frac{\mu}{8} \Vert \nabla\cdot \overline{e}_u^{n-1/2 }\Vert^2.
\end{equation*}
In addition,
\begin{equation*}
	\begin{aligned}
		%%%%%%%%%%%%%%%%%%%%%%%%%%
		\mathbb{H}_4\leq &\textstyle~ C_{\text{Pi}} C_{\text{st}} \Vert \nabla \bu \Vert_{L^\infty(0,T;L^\infty)}   \Vert \bu(t_{n-1/2}) - \overline{\bu}(t_{n-1/2}) \Vert_1 \Vert \overline{e}_u^{n-1/2} \Vert\\
		\leq &\textstyle~ \frac{1}{32}\Vert e_u^n\Vert^2 +  \frac{1}{32}\Vert e_u^{n-1}\Vert^2 + C_{\text{Pi}}^2 C_{\text{st}}^2  \Vert \nabla \bu\Vert_{L^\infty(0,T;L^\infty)}^2\frac{\tau^3}{6}   \Vert \bu_{tt}\Vert_{L^2(t_{n-1},t_n;H^1)}^2, \\
		%%%%%%%%%%%%%%%%%%%%%%%%%%
		\mathbb{H}_5\leq &\textstyle~ C_{\text{Pi}}  C_{\text{st}} \Vert \nabla \bu \Vert_{L^\infty(0,T;L^\infty)}  \Vert \overline{\bu}(t_{n-1/2})-{\rm P}_{\text{st}} \overline{\bu}(t_{n-1/2}) \Vert_1  \Vert  \overline{e}_u^{n-1/2}\Vert\\
		\leq & \textstyle ~ \frac{1}{32}\Vert e_u^n\Vert^2 +  \frac{1}{32}\Vert e_u^{n-1}\Vert^2 +  4  h^4C_{\text{Pi}}^2 C_{\text{st}}^4 \Vert \nabla \bu \Vert_{L^\infty(0,T;L^\infty)}^2 \Vert \bu \Vert_{L^\infty(0,T;H^3)}^2.
	\end{aligned}
\end{equation*}
Combining all the above estimates on $\mathbb{H}_i,\ i=1,\dots,5$, we have
\begin{equation}
	\begin{aligned}
		\mathbb{I}_4 + \mathbb{I}_5 &\leq \textstyle ~
		\left(\frac18 +  \frac12 C_{\text{st}} \Vert \nabla \bu \Vert_{L^\infty(0,T;L^\infty)} + \frac{1}{4\mu}C_{\text{Pi}}^2 C_{\text{st}}^2 \Vert \nabla \bu \Vert_{L^\infty(0,T;L^\infty)}^2  \right) \left( \Vert e_u^n\Vert^2 +  \Vert e_u^{n-1}\Vert^2 \right)\\
		&\textstyle +\frac{\tau^3}{12}C_{\text{em}}^4  \Vert \bu_{tt}\Vert_{L^2(t_{n-1},t_n;H^1)}^2\Vert \bu \Vert_{L^\infty(0,T;H^2)}^2
		+ C_{\text{Pi}}^2 C_{\text{st}}^2  \Vert \nabla \bu\Vert_{L^\infty(0,T;L^\infty)}^2\frac{\tau^3}{6}   \Vert \bu_{tt}\Vert_{L^2(t_{n-1},t_n;H^1)}^2
		\\
		%+  20C_{\text{em}}^4 C_{\text{st}}^2 h^4 \Vert \bu \Vert_{L^\infty(0,T;H^3)}^2\Vert \bu \Vert_{L^\infty(0,T;H^2)}^2
		&\textstyle
		+  2C_{\text{em}}^4 C_{\text{st}}^2 h^4 \Vert \bu \Vert_{L^\infty(0,T;H^3)}^2\Vert \bu \Vert_{L^\infty(0,T;H^2)}^2
		+ 4  h^4C_{\text{Pi}}^2 C_{\text{st}}^4 \Vert \nabla \bu \Vert_{L^\infty(0,T;L^\infty)}^2 \Vert \bu \Vert_{L^\infty(0,T;H^3)}^2\\
		&\textstyle 
		+  \frac{\mu}{8} \Vert \nabla\cdot \overline{e}_u^{n-1/2 }\Vert^2.
	\end{aligned}
\end{equation}
\textbf{Part II.} In this part, we focus on bounding $\mathbb{I}_7$ to recover the optimal approximation error of $\mathcal{O}(\varepsilon^2)$. The primary difficulty stems from the term $\overline{e}_p^{n-1/2}$, which must be coupled with $\varepsilon$. A direct application of Young's inequality to this term would yield merely a sub-optimal estimate of $\mathcal{O}(\varepsilon)$. To overcome this, we first decompose $\mathbb{I}_7$ into two terms by introducing the Lagrange interpolant of the time derivative of pressure. Exploiting the properties of the Scott-Vogelius element, we can then relate this interpolant to the velocity space. Consequently, we invoke the momentum equation to circumvent the direct estimation of $\overline{e}_p^{n-1/2}$.
 As for $\mathbb{I}_7$, we have
\begin{equation*}
	\begin{aligned}
		\textstyle
		\varepsilon \left(\frac{\partial p(t_{n-1/2})}{\partial t}, \overline{e}_p^{n-1/2}\right) =& \textstyle \ \varepsilon \left(\frac{\partial p(t_{n-1/2})}{\partial t} -I_h \frac{\partial p(t_{n-1/2})}{\partial t} , \overline{e}_p^{n-1/2}\right) + \varepsilon \left( I_h\frac{\partial p(t_{n-1/2})}{\partial t}, \overline{e}_p^{n-1/2}\right)\\
		= &\textstyle \ \varepsilon \left(\frac{\partial p(t_{n-1/2})}{\partial t} -I_h \frac{\partial p(t_{n-1/2})}{\partial t} , \overline{e}_p^{n-1/2}\right) + \varepsilon \left( \nabla\cdot \bw_h^{n-1/2}, \overline{e}_p^{n-1/2}\right),
	\end{aligned}
\end{equation*}
where
\begin{equation}\label{pressure:int}
	\begin{aligned}
		\textstyle 
		\varepsilon \left(\frac{\partial p(t_{n-1/2})}{\partial t} -I_h \frac{\partial p(t_{n-1/2})}{\partial t} , \overline{e}_p^{n-1/2}\right)	\leq \frac{\varepsilon}{8}\left(\Vert e_p^n\Vert^2 + \Vert  e_p^{n-1}\Vert^2\right) 
		+\varepsilon C_0^2 h^4 \Vert  \frac{\partial p(t_{n-1/2})}{\partial t} \Vert_2^2,
	\end{aligned}
\end{equation} 
and,
\begin{equation*}
	\begin{aligned}
		& \textstyle 
		\varepsilon\left( \nabla\cdot \bw_h^{n-1/2}, \overline{e}_p^{n-1/2}\right) =  \textstyle	\varepsilon\left(\frac{e_u^n-e_u^{n-1}}{\tau}, \bw_h^{n-1/2}\right)+\nu\varepsilon\left(\nabla \overline{e}_u^{n-1/2}, \nabla \bw_h^{n-1/2}\right)
		+\mu\varepsilon\left(\nabla \cdot \overline{e}_u^{n-1/2}, \nabla \cdot \bw_h^{n-1/2}\right) \\
		& \qquad\quad  \textstyle + \varepsilon\left(\frac{\partial \bu\left(t_{n-1/2} \right)}{\partial t}-{\rm P}_{\text{st}} \frac{\bu\left(t_n\right)-\bu\left(t_{n-1}\right)}{\tau}, \bw_h^{n-1/2}\right) 
		+\nu\varepsilon\left(\nabla{\rm P}_{\text{st}}\left( \bu\left(t_{n-1/2}\right)-\overline{\bu}\left(t_{n-1/2}\right)\right), \nabla \bw_h^{n-1/2}\right)\\
		& \qquad\quad  \textstyle - \varepsilon\left( {\rm Q}_{\text{st}}\left( p\left(t_{n-1/2}\right)-\overline{p}\left(t_{n-1/2}\right) \right), \nabla \cdot \bw_h^{n-1/2}\right)\\
		& \qquad\quad  \textstyle +\varepsilon c\left(\bu\left(t_{n-1/2}\right), \bu\left(t_{n-1/2}\right), \bw_h^{n-1/2}\right)  -\varepsilon c\left(\overline{\bu}_h^{n-1 / 2}, \overline{\bu}_h^{n-1/2}, \bw_h^{n-1/2}\right)
		:=\sum\limits_{i=1}^8\mathbb{J}_i,
	\end{aligned}
\end{equation*}
where 
\begin{equation}
	\left\{
	\begin{aligned}
		\textstyle 
		\mathbb{J}_2 \leq &~ \frac{\nu}{8} \Vert \nabla \overline{e}_u^{n-1/2}\Vert^2 + 2\nu \varepsilon^2 C_{\text{inf}}^2\Vert \frac{\partial p(t_{n-1/2})}{\partial t} \Vert^2,\\
		\textstyle 
		\mathbb{J}_3  \leq &~ \frac{\mu}{8} \Vert \nabla \cdot \overline{e}_u^{n-1/2}\Vert^2 + 2\mu \varepsilon^2 C_{\text{inf}}^2\Vert \frac{\partial p(t_{n-1/2})}{\partial t} \Vert^2.
	\end{aligned}
	\right.
\end{equation}
The last five terms of the right-hand side can be estimated silimar to $\mathbb{I}_i,\ i=1,\dots,5$, where $\bw_h^{n-1/2}$ plays the role of $\overline{e}_u^{n-1/2}$. Then, we have
\begin{equation}\label{time:cn}
	\begin{aligned}
		\sum\limits_{i=4}^8 \mathbb{J}_i&  \leq   \textstyle ~   \varepsilon^2C_{\text{Pi}}^2C_{\text{inf}}^2 \left( \frac12 +   C_{\text{st}} \Vert \nabla \bu \Vert_{L^\infty(0,T;L^\infty)} + \frac{1}{2\mu}C_{\text{Pi}}^2 C_{\text{st}}^2 \Vert \nabla \bu \Vert_{L^\infty(0,T;L^\infty)}^2 \right)  \Vert  \frac{\partial p(t_{n-1/2})}{\partial t} \Vert^2 + \frac{3\varepsilon^2}{8}C_{\text{inf}}^2 \Vert \frac{\partial p(t_{n-1/2})}{\partial t}\Vert^2
		 \\
		&\textstyle  + \frac{\tau^3}{640} \Vert \bu_{ttt}\Vert_{L^2(t_{n-1}, t_n;L^2)}^2 +  \frac{2C_{\text{st}}^2h^6 }{\tau} \Vert \bu_t\Vert_{L^2(t_{n-1}, t_n;H^3)}^2+ \frac{\nu C_{\text{st}}^2\tau^3}{24} \Vert \bu_{tt}\Vert_{L^2(t_{n-1}, t_n;H^1)}^2
		+ \frac{ C_{\text{st}}^2\tau^3}{24 \mu} \Vert p_{tt}\Vert_{L^2(t_{n-1}, t_n;L^2)}^2\\
		&\textstyle +\frac{\tau^3}{12}C_{\text{em}}^4  \Vert \bu_{tt}\Vert_{L^2(t_{n-1},t_n;H^1)}^2\Vert \bu \Vert_{L^\infty(0,T;H^2)}^2
		+ C_{\text{Pi}}^2 C_{\text{st}}^2  \Vert \nabla \bu\Vert_{L^\infty(0,T;L^\infty)}^2\frac{\tau^3}{6}   \Vert \bu_{tt}\Vert_{L^2(t_{n-1},t_n;H^1)}^2
		\\
		&\textstyle
		+  2C_{\text{em}}^4 C_{\text{st}}^2 h^4 \Vert \bu \Vert_{L^\infty(0,T;H^3)}^2\Vert \bu \Vert_{L^\infty(0,T;H^2)}^2
		+ 4  h^4C_{\text{Pi}}^2 C_{\text{st}}^4 \Vert \nabla \bu \Vert_{L^\infty(0,T;L^\infty)}^2 \Vert \bu \Vert_{L^\infty(0,T;H^3)}^2.
	\end{aligned}
\end{equation}
Combining all the estimates from \eqref{cn:1st} to \eqref{time:cn}, multiplying $2\tau$ on both sides, and summing up from $n=1$ to $n$, there holds
\begin{equation}\label{cn:all}
	\begin{aligned}
		&\textstyle	 \Vert e_u^n \Vert^2 + \varepsilon   \Vert e_p^n \Vert^2   + \nu \tau\sum\limits_{i=1}^n \Vert \nabla\overline{e}_u^{i-1/2} \Vert^2 + \mu  \tau\sum\limits_{i=1}^n \Vert \nabla\cdot \overline{e}_u^{i-1/2} \Vert^2\\
		& \textstyle 
		\leq ~ \textstyle
		 \tau\sum\limits_{i=1}^n \left(1 +  2 C_{\text{st}} \Vert \nabla \bu \Vert_{L^\infty(0,T;L^\infty)} +\frac{1}{ \mu}C_{\text{Pi}}^2C_{\text{st}}^2 \Vert \nabla\bu \Vert_{L^\infty(0,T;L^\infty)}  \right)   \Vert e_u^i\Vert^2
		+ \frac{ \tau \varepsilon}{2}\sum\limits_{i=1}^n   \Vert e_p^i\Vert^2 
		\\
			&\textstyle ~
			+  2\tau\sum\limits_{i=1}^n\varepsilon\left(\frac{e_u^i-e_u^{i-1}}{\tau}, \bw_h^{i-1/2}\right)  
		+ \frac{\tau^4}{160} \Vert \bu_{ttt}\Vert_{L^2(0, t_n;L^2)}^2 +  8C_{\text{st}}^2h^6  \Vert \bu_t\Vert_{L^2(0, t_n;H^3)}^2\\
		&\textstyle ~  + \frac{\nu C_{\text{st}}^2\tau^4}{6} \Vert \bu_{tt}\Vert_{L^2(0, t_n;H^1)}^2  + \frac{ C_{\text{st}}^2\tau^4 }{6\mu } \Vert p_{tt}\Vert_{L^2(0, t_n;L^2)}^2
		+\frac{\tau^4}{3}C_{\text{em}}^4  \Vert \bu_{tt}\Vert_{L^2(0,t_n;H^1)}^2\Vert \bu \Vert_{L^\infty(0,T;H^2)}^2
		\\
		&\textstyle ~ 
			+  \frac{2\tau^4}{3} C_{\text{Pi}}^2 C_{\text{st}}^2 \Vert \nabla \bu\Vert_{L^\infty(0,T;L^\infty)}^2   \Vert \bu_{tt}\Vert_{L^2(0,t_n;H^1)}^2
		+  8t_nC_{\text{em}}^4 C_{\text{st}}^2 h^4 \Vert \bu \Vert_{L^\infty(0,T;H^3)}^2\Vert \bu \Vert_{L^\infty(0,T;H^2)}^2	\\
			&\textstyle~ 
			+ 16t_n  h^4C_{\text{Pi}}^2 C_{\text{st}}^4 \Vert \nabla \bu \Vert_{L^\infty(0,T;L^\infty)}^2 \Vert \bu \Vert_{L^\infty(0,T;H^3)}^2
    + 
\frac{\varepsilon\tau^4}{320} \Vert p_{ttt}\Vert_{L^2(0, t_n;L^2)}^2 +  4\varepsilon C_{\text{st}}^2h^4  \Vert p_t\Vert_{L^2(0, t_n;H^2)}^2 			\\
		&\textstyle ~
	+ 2\tau\sum\limits_{i=1}^n \varepsilon C_0^2 h^4 \Vert  \frac{\partial p(t_{i-1/2})}{\partial t} \Vert_2^2  
	 +   \tau\sum\limits_{i=1}^n  \left(4 \left( \nu+\mu\right)+\frac34 \right) \varepsilon^2 C_{\text{inf}}^2\Vert \frac{\partial p(t_{i-1/2})}{\partial t} \Vert^2 \\
	&\textstyle  ~  +   2\tau\sum\limits_{i=1}^n\varepsilon^2C_{\text{Pi}}^2C_{\text{inf}}^2 \left(\frac12 +   C_{\text{st}} \Vert \nabla \bu \Vert_{L^\infty(0,T;L^\infty)} +\frac{1}{2\mu} C_{\text{Pi}}^2 C_{\text{st}}^2\Vert \bu \Vert_{L^\infty(0,T;L^\infty)}  \right)  \Vert  \frac{\partial p(t_{i-1/2})}{\partial t} \Vert^2,
	\end{aligned}
\end{equation}
where $\Vert e_u^0 \Vert^2=0$ and $ \Vert  e_p^0 \Vert^2=0$ are applied, and 
\begin{equation*}
	\begin{aligned}
		\textstyle 
		2\tau\sum\limits_{i=1}^n\varepsilon\left(\frac{e_u^i-e_u^{i-1}}{\tau}, \bw_h^{i-1/2}\right) & =  \textstyle 2\varepsilon \left(e_u^n, \bw_h^{n-1/2}\right) -2\tau \varepsilon \sum\limits_{i=1}^{n-1} \left(e_u^i, \frac{\bw_h^{i+1/2} - \bw_h^{i-1/2}}{\tau}\right)\\
		&\leq \textstyle \ \frac14 \Vert e_u^n \Vert^2 + 4\varepsilon^2 \Vert \bw_h^{n-1/2}\Vert^2 + \tau\sum\limits_{i=1}^{n-1} \frac14  \Vert e_u^i \Vert^2 + \tau\sum\limits_{i=1}^{n-1} 4\varepsilon^2 \Vert \frac{\bw_h^{i+1/2} - \bw_h^{i-1/2}}{\tau}\Vert^2 \\
		&\leq \textstyle \ \frac14 \Vert e_u^n \Vert^2 + 4\varepsilon^2C_{\text{Pi}}^2C_{\text{inf}}^2 \Vert \frac{\partial p(t_{n-1/2})}{\partial t}\Vert^2 + \tau\sum\limits_{i=1}^{n-1} \frac14  \Vert e_u^i \Vert^2 +   4\varepsilon^2  C_{\text{Pi}}^2C_{\text{inf}}^2 \Vert p_{tt}\Vert_{L^2(t_0, t_{n};L^2)}^2. 
	\end{aligned}
\end{equation*}
Lastly, applying  $\tau\leq \frac{1}{ 4 + 8 C_{\text{st}} \Vert \nabla \bu \Vert_{L^\infty(0,T;L^\infty)} +\frac{4}{ \mu}C_{\text{Pi}}^2C_{\text{st}}^2 \Vert \bu \Vert_{L^\infty(0,T;L^\infty)}   }:=\tau_1$ and Gr\"owall's inequality with $\max\{\tau, h, \nu, \mu\}\leq 1$, we have
\begin{equation}\label{cn:last}
	\begin{aligned}
		&\textstyle	 \Vert e_u^n \Vert^2 + \varepsilon   \Vert e_p^n \Vert^2   + \nu \tau\sum\limits_{i=1}^n \Vert \nabla\overline{e}_u^{i-1/2} \Vert^2 + \mu  \tau\sum\limits_{i=1}^n \Vert \nabla\cdot \overline{e}_u^{i-1/2} \Vert^2\\
		& \textstyle 
		\leq ~ \textstyle
		\exp\left( t_n\left(\frac52+  4 C_{\text{st}} \Vert \nabla \bu \Vert_{L^\infty(0,T;L^\infty)} +\frac{2}{ \mu} C_{\text{Pi}}^2C_{\text{st}}^2\Vert \nabla\bu \Vert_{L^\infty(0,T;L^\infty)} \right) \right) \left\{
		\frac{\tau^4}{80} \Vert \bu_{ttt}\Vert_{L^2(0, t_n;L^2)}^2 +  16C_{\text{st}}^2h^6  \Vert \bu_t\Vert_{L^2(0, t_n;H^3)}^2
		\right.\\
		&\textstyle ~  + \frac{\nu C_{\text{st}}^2\tau^4}{3} \Vert \bu_{tt}\Vert_{L^2(0, t_n;H^1)}^2  + \frac{ C_{\text{st}}^2\tau^4 }{3\mu } \Vert p_{tt}\Vert_{L^2(0, t_n;L^2)}^2
		+\frac{2\tau^4}{3}C_{\text{em}}^4  \Vert \bu_{tt}\Vert_{L^2(0,t_n;H^1)}^2\Vert \bu \Vert_{L^\infty(0,T;H^2)}^2
		\\
		&\textstyle ~ 
		+  \frac{4\tau^4}{3} C_{\text{Pi}}^2 C_{\text{st}}^2 \Vert \nabla \bu\Vert_{L^\infty(0,T;L^\infty)}^2   \Vert \bu_{tt}\Vert_{L^2(0,t_n;H^1)}^2
		+  16 t_nC_{\text{em}}^4 C_{\text{st}}^2 h^4 \Vert \bu \Vert_{L^\infty(0,T;H^3)}^2\Vert \bu \Vert_{L^\infty(0,T;H^2)}^2	\\
		&\textstyle~ 
		+ 32 t_n  h^4C_{\text{Pi}}^2 C_{\text{st}}^4 \Vert \nabla \bu \Vert_{L^\infty(0,T;L^\infty)}^2 \Vert \bu \Vert_{L^\infty(0,T;H^3)}^2
		+ 
		\frac{\varepsilon\tau^4}{160} \Vert p_{ttt}\Vert_{L^2(0, t_n;L^2)}^2 +  8 \varepsilon C_{\text{st}}^2h^4  \Vert p_t\Vert_{L^2(0, t_n;H^2)}^2 			\\
		&\textstyle ~
		+ 4\tau\sum\limits_{i=1}^n \varepsilon C_0^2 h^4 \Vert  \frac{\partial p(t_{i-1/2})}{\partial t} \Vert_2^2  
		+  2 \tau\sum\limits_{i=1}^n  \left(4 \left( \nu+\mu\right)+\frac34 \right) \varepsilon^2 C_{\text{inf}}^2\Vert \frac{\partial p(t_{i-1/2})}{\partial t} \Vert^2 \\
		&\textstyle  ~  \left. +   4\tau\sum\limits_{i=1}^n\varepsilon^2C_{\text{Pi}}^2C_{\text{inf}}^2 \left(\frac12 +   C_{\text{st}} \Vert \nabla \bu \Vert_{L^\infty(0,T;L^\infty)} +\frac{1}{2\mu} C_{\text{Pi}}^2 C_{\text{st}}^2\Vert \bu \Vert_{L^\infty(0,T;L^\infty)}  \right)  \Vert  \frac{\partial p(t_{i-1/2})}{\partial t} \Vert^2 
		\right\}\\
		&
		\leq ~ \textstyle
	\exp\left( t_n\left(\frac52+  4 C_{\text{st}} \Vert \nabla \bu \Vert_{L^\infty(0,T;L^\infty)} +\frac{2}{ \mu} C_{\text{Pi}}^2C_{\text{st}}^2\Vert \nabla\bu \Vert_{L^\infty(0,T;L^\infty)} \right) \right) \mathcal{C}(t_n)(\tau^4 + h^4 + \varepsilon^2).
	\end{aligned}
\end{equation}

\end{proof}

\begin{remark}
	If we consider the linearized scheme \eqref{acm:gdiv:lin}, $\mathbb{I}_4 + \mathbb{I}_5$ will be 
	\begin{equation*}
			\begin{aligned}
				\mathbb{I}_4 + \mathbb{I}_5 = &~ -c(\bu(t_{n-1/2})-\widetilde{\bu}(t_{n-1/2}), \bu(t_{n-1/2}), \overline{e}_u^{n-1/2}) -c(\widetilde{\bu}(t_{n-1/2})-{\rm P}_{\text{st}} \widetilde{\bu}(t_{n-1/2}), \bu(t_{n-1/2}), \overline{e}_u^{n-1/2})\\
				&~ -c(\widetilde{e}_u^{n-1/2}, {\rm P}_{\text{st}}\overline{\bu}(t_{n-1/2}), \overline{e}_u^{n-1/2})  - c( {\rm P}_{\text{st}} \widetilde{\bu}(t_{n-1/2}), \bu(t_{n-1/2}) - \overline{\bu}(t_{n-1/2}), \overline{e}_u^{n-1/2}) \\
				&~ - c({\rm P}_{\text{st}} \widetilde{\bu}(t_{n-1/2}), \overline{\bu}(t_{n-1/2})-{\rm P}_{\text{st}} \overline{\bu}(t_{n-1/2}), \overline{e}_u^{n-1/2}) 
				:= \sum\limits_{i=1}^5 \widetilde{\mathbb{H}}_i,
			\end{aligned}
	\end{equation*}
	where ${\mathbb{H}}_1, {\mathbb{H}}_2, {\mathbb{H}}_4, \text{ and }{\mathbb{H}}_5$ can be estimated in the same manner. Meanwhile, $\widetilde{\mathbb{H}}_3$ is estimated as follow
	\begin{equation*}
		\begin{aligned}
		\textstyle \widetilde{\mathbb{H}}_3 & = ( (\widetilde{e}_u^{n-1/2}\cdot\nabla) {\rm P}_{\text{st}}\overline{\bu}(t_{n-1/2}), \overline{e}_u^{n-1/2}) + \frac12 ((\nabla\cdot \widetilde{e}_u^{n-1/2} ){\rm P}_{\text{st}}\overline{\bu}(t_{n-1/2}),  \overline{e}_u^{n-1/2}) \\
		&\textstyle \leq \Vert \nabla {\rm P}_{\text{st}}\overline{\bu}(t_{n-1/2})\Vert_{L^\infty}  \Vert \widetilde{e}_u^{n-1/2} \Vert \Vert \overline{e}_u^{n-1/2}\Vert  + \frac12  \Vert {\rm P}_{\text{st}}\overline{\bu}(t_{n-1/2}) \Vert_{L^\infty} \Vert \nabla\cdot \widetilde{e}_u^{n-1/2} \Vert \Vert \overline{e}_u^{n-1/2} \Vert.
		\end{aligned}
	\end{equation*} 
	Then, by Young's inequality, we can still avoid introducing $\frac{1}{\nu}$ on the right-hand side. Therefore, Reynolds number robustness can also be achieved.
\end{remark}
%\begin{remark}
%	This part, we show the Reynolds number robust error estimate. Thanks to the introduced grad-div stabilization term, we can avoid $\frac{1}{\nu}$ in  $\mathbb{I}_3$, $\mathbb{H}_3$, and $\mathbb{J}_3$. 
%\end{remark}

\section{Optimal $L^2$ error estimate on ACM-CN}\label{sect:opl2}
In this section, we aim to establish the optimal $L^2$ error estimate for the velocity. Usually, there are two approaches: the negative norm technique \cite{he2018priori,guo2018} and the Aubin-Nitsche duality argument. However, both of them fail for the ACM scheme due to Eq. \eqref{acm:rog}. The negative norm technique relies crucially on the incompressibility property, as demonstrated in \cite[Lemma 4.2]{he2018priori}. Meanwhile, the duality argument cannot circumvent the estimate \eqref{pressure:int} while preserving the optimal perturbation rate of $\mathcal{O}(\varepsilon^2)$. To achieve the optimal $L^2$ error estimate, we instead consider the difference between the solutions of the classical Navier-Stokes solver and the CN-ACM method.

Since the Scott-Vogelius element is pointwisely divergence-free, the term $(\nabla\cdot\bm{\xi}_h, \nabla\cdot\bv_h)$ vanishes identically. However, to maintain consistency with the grad-div stabilized ACM formulation, we retain this term in the scheme. Consequently, the classical Navier-Stokes finite element discretization reads: given $(\bxi_h^{n-1}, r_h^{n-1} )\in \bX_h\times M_h$, find $ (\bxi_h^n, r_h^n)\in \bX_h\times M_h$, such that
\begin{subequations}\label{NS:gdiv:cn:eq}
	\begin{empheq}[left=\empheqlbrace]{align}
		& \textstyle
		\left(\frac{\bm{\xi}_h^n-\bm{\xi}_h^{n-1}}{\tau}, \bv_h\right)+\nu\left(\nabla \overline{\bm{\xi}}_h^{n-1/2}, \nabla \bv_h\right) +\mu\left(\nabla\cdot\overline{\bm{\xi}}_h^{n-1/2}, \nabla\cdot\bv_h\right) \notag \\
		&\textstyle    
		\qquad\qquad\qquad+c\left( \overline{\bm{\xi}}_h^{n-1/2}, \overline{\bm{\xi}}_h^{n-1/2}, \bv_h\right)
		-\left( \overline{r}_h^{n-1/2}, \nabla \cdot \bv_h\right)=\left(\bm f^{n-1/2}, \bv_h\right), \forall\ \bv_h\in \bX_h, \label{NS:mom:eq:cn}\\
		& \textstyle
		\left( \nabla\cdot  \overline{\bm{\xi}}_h^{n-1/2}, q_h\right)=0, \forall \ q_h\in M_h,\label{NS:pres:eq:cm}
	\end{empheq}
\end{subequations}
As the finite element scheme \eqref{NS:gdiv:cn:eq} has been thoroughly investigated, we simply state its error estimate and omit the proof. For the optimal spatial estimate, we refer the reader to \cite{heywood1990finite}; for the optimal temporal estimate, see \cite{rang2008pressure,sonner2020second}. Most importantly, this scheme is also robust with respect to the Reynolds number; see \cite{john2017divergence,olshanskii2020longer}.
\begin{lem}\label{op:reg:ns}
	There exists a positive constant $K(T)$, independent of $\nu$, such that
	\begin{equation}
	\left\{
		\begin{aligned}
	\nu\sum\limits_{i=1}^n\tau \Vert \nabla(\bm{\xi}_h^i-{\rm P}_{\text{st}}\bu(t_i))\Vert^2  +\sum\limits_{i=1}^n\tau \Vert  r_h^i-{\rm Q}_{\text{st}}p(t_i) \Vert^2\leq K(T)(\tau^4+h^4),\\
		\Vert\bm{\xi}_h^n-{\rm P}_{\text{st}}\bu(t_n)\Vert^2 \leq K(T)(\tau^4+h^6).
	\end{aligned}
	\right.
	\end{equation}
\end{lem}
Hence, if $\tau\leq h^{\frac{1}{2}+\frac{d}{4}}$ with $h\leq 1$, then
$$\Vert \nabla(\bm{\xi}_h^n-{\rm P}_{\text{st}}\bu(t_n)) \Vert_{L^{\infty}}\leq C_{\text{inv}}h^{-1-\frac{d}{2}} \Vert  \bm{\xi}_h^n-{\rm P}_{\text{st}}\bu(t_n) \Vert\leq 2 C_{\text{inv}} \sqrt{K(T)},$$
and
$$\Vert\nabla\bm{\xi}_h^n\Vert_{L^{\infty}}\leq\Vert\nabla{\rm P}_{\text{st}}\bu(t_n)\Vert_{L^{\infty}}+ 2 C_{\text{inv}}\sqrt{K(T)}:= C_{5},$$
which is boudned.

Now, we turn to estimate the approximation between CN-ACM and Eq. \eqref{NS:gdiv:cn:eq}. First, we introduce the difference quotient $\partial_\tau \phi^n:=\frac{\phi^n-\phi^{n-1}}{\tau}$. Subsequently, 
we establish a sub-optimal estimate for the difference quotient of the pressure error, namely $\partial_\tau (r_h^n-{\rm{Q}_{\text{st}}}p(t_n)) = \frac{(r_h^n-{\rm{Q}_{\text{st}}}p(t_n))  -  (r_h^{n-1}-{\rm{Q}_{\text{st}}}p(t_{n-1})) }{\tau}$. Despite being sub-optimal, this estimate is sufficient to bound the second-order difference quotient $\frac{r_h^n-2r_h^{n-1}+r_h^{n-2}}{\tau^2}$. Moreover, it is robust with respect to the Reynolds number. For simplicity, $\widetilde{C}$ denotes a generic constant independent of $h$, $\tau$, and $\nu^{-1}$, which may change from line to line.
Set $E_u^n={\rm P}_{\text{st}}\bu(t_n) - \bxi_h^n$ and $E_p^n={\rm Q}_{\text{st}}p(t_n) - r_h^n$.
\begin{thm}\label{approximation}
	Given $h\leq 1$, if $\textstyle \tau\leq \min\{h^{\frac12 + \frac{d}{4}}, \frac14 \left(\frac{C_{5}^2}{4\mu}+ C_{5} + \frac54 \right)^{-1}\}$, we can have the following estimate
	\begin{equation}
		\textstyle\frac12 \Vert E_u^n \Vert^2   + \varepsilon \Vert E_p^n \Vert^2  +  \nu \sum\limits_{i=1}^n \tau \Vert\nabla \overline{E}_u^{i-1/2}\Vert^2 \leq  \widetilde{C} \varepsilon^2.
	\end{equation}
\end{thm}

\begin{proof}
To have the difference between $\bm{\xi}_h^n$ and $\bu_h^n$, we define $E_u^n=\bm{\xi}_h^n-\bu_h^n$, and $E_p^n=r_h^n-p_h^n$, then
\begin{subequations}\label{eq:gdiv:be:eq:l2}
	\begin{empheq}[left=\empheqlbrace]{align}
		& \textstyle
		\left(\frac{E_u^n-E_u^{n-1}}{\tau}, \bv_h\right)+\nu\left(\nabla \overline{E}_u^{n-1/2}, \nabla \bv_h\right) +\mu\left(\nabla\cdot \overline{E}_u^{n-1/2}, \nabla\cdot\bv_h\right)\notag\\
		&\textstyle 
		\qquad  
		+ c(\overline{\bm{\xi}}_h^{n-1/2}, \overline{\bm{\xi}}_h^{n-1/2}, \bv_h)   
		- c\left( \overline{\bu}_h^{n-1/2}, \overline{\bu}_h^{n-1/2}, \bv_h\right) 
		-\left( \overline{E}_p^{n-1/2}, \nabla \cdot \bv_h\right)=0 , \label{eq:mom:eq:l2}\\
		& \textstyle
		\varepsilon \left( \frac{E_p^n-E_p^{n-1}}{\tau}, q_h\right) + \left( \nabla\cdot \overline{E}_u^{n-1/2}, q_h\right)	= \varepsilon\left(\frac{r_h^n - r_h^{n-1}}{\tau}, q_h\right).\label{eq:pres:eq:l2}
	\end{empheq}
\end{subequations}

Let $\bv_h= \overline{E}_u^{n-1/2} $ and $q_h= \overline{E}_p^{n-1/2}$,  there holds
\begin{equation}\label{l2}
	\begin{aligned}
		&\textstyle	 \frac{\Vert E_u^n \Vert^2 -   \Vert E_u^{n-1} \Vert^2}{2\tau}   + \varepsilon  \frac{\Vert E_p^n \Vert^2 -   \Vert E_p^{n-1} \Vert^2}{2\tau}+ \nu\Vert\nabla \overline{E}_u^{n-1/2}\Vert^2 +\mu\Vert\nabla\cdot \overline{E}_u^{n-1/2} \Vert^2\\
		&
		= ~ \textstyle
		- c(\overline{\bm{\xi}}_h^{n-1/2}, \overline{\bm{\xi}}_h^{n-1/2},  \overline{E}_u^{n-1/2})   
		+ c\left( \overline{\bu}_h^{n-1/2},  \overline{\bu}_h^{n-1/2},  \overline{E}_u^{n-1/2}\right)+\varepsilon\left(\frac{r_h^n - r_h^{n-1}}{\tau},  \overline{E}_p^{n-1/2}\right) 
		:=\sum\limits_{i=1}^{3}\mathbb{K}_i.
	\end{aligned}
\end{equation}
where by the skew-symmetric property of the trilinear form $c(\cdot,\cdot,\cdot)$ and integration-by-parts, we have
\begin{equation}\label{nonlinear:l2}
	\begin{aligned}
		\mathbb{K}_1 + \mathbb{K}_2 &\textstyle =   	- c(\overline{\bm{\xi}}_h^{n-1/2}, \overline{\bm{\xi}}_h^{n-1/2},  \overline{E}_u^{n-1/2})   
		+ c\left( \overline{\bu}_h^{n-1/2},  \overline{\bu}_h^{n-1/2},  \overline{E}_u^{n-1/2}\right) \\
		&\textstyle = - c(\overline{E}_u^{n-1/2}, \overline{\bm{\xi}}_h^{n-1/2},  \overline{E}_u^{n-1/2})  \\
		&\textstyle = -\frac{1}{2}\left((\overline{E}_u^{n-1/2}\cdot\nabla) \overline{\bm{\xi}}_h^{n-1/2}, \overline{E}_u^{n-1/2} \right)+\frac{1}{2}\left((\overline{E}_u^{n-1/2}\cdot\nabla)\overline{E}_u^{n-1/2} ,\overline{\bm{\xi}}_h^{n-1/2}  \right)
		\\
		&\textstyle = -\frac{1}{2}\left(\nabla\cdot\overline{E}_u^{n-1/2} \overline{E}_u^{n-1/2}, \overline{\bm{\xi}}_h^{n-1/2} \right)-
		\left((\overline{E}_u^{n-1/2}\cdot\nabla) \overline{\bm{\xi}}_h^{n-1/2}, \overline{E}_u^{n-1/2} \right)\\
		&\textstyle \leq \frac{1}{2} C_{5}\Vert\nabla\cdot \overline{E}_u^{n-1/2} \Vert\Vert \overline{E}_u^{n-1/2} \Vert  + C_5\Vert  \overline{E}_u^{n-1/2} \Vert^2\\
		&\textstyle \leq \frac{\mu}{4}\Vert\nabla\cdot \overline{E}_u^{n-1/2} \Vert^2+\left(\frac{C_{5}^2}{4\mu}+C_{5}\right)\Vert \overline{E}_u^{n-1/2} \Vert ^2.
	\end{aligned}
\end{equation}

As for $\mathbb{K}_3$, there exists $ \psi_h^{n-1}\in \bX_h$, $\forall q_h\in M_h$, such that  
\begin{equation*}
	\textstyle\varepsilon\left(\frac{r_h^n - r_h^{n-1}}{\tau},   q_h \right)= \varepsilon\left(\nabla\cdot \psi_h^{n-1/2},  q_h \right), \text{ and } \textstyle \Vert\nabla \psi_h^{n-1/2}\Vert\leq C_{\text{inf}} \left\Vert \frac{r_h^n - r_h^{n-1}}{\tau} \right\Vert.
\end{equation*}
There also holds 
$	\textstyle \Vert\nabla \frac{\psi_h^{n-1/2} - \psi_h^{n-3/2}}{\tau}\Vert\leq C_{\text{inf}} \left\Vert  \frac{ r_h^n - 2r_h^{n-1} +   r_h^{n-2} }{\tau^2}\right\Vert.$ 
Meanwhile, 
\begin{equation*}
	\begin{aligned}
		\textstyle 
		\varepsilon\left( \nabla\cdot  \psi_h^{n-1/2}, \overline{E}_p^{n-1/2}\right) &=  \textstyle	\varepsilon\left(\frac{E_u^n-E_u^{n-1}}{\tau},  \psi_h^{n-1/2}\right)+\nu\varepsilon\left(\nabla \overline{E}_u^{n-1/2}, \nabla  \psi_h^{n-1/2}\right) + \varepsilon \mu \left(\nabla \cdot \overline{E}_u^{n-1/2}, \nabla \cdot  \psi_h^{n-1/2}\right) 
		\\
		&   \quad\textstyle
		 \textstyle +\varepsilon c\left(\overline{\bm{\xi}}_h^{n-1/2}, \overline{\bm{\xi}}_h^{n-1/2}, \bw_h^{n-1/2}\right)  -\varepsilon c\left(\overline{\bu}_h^{n-1 / 2}, \overline{\bu}_h^{n-1/2},  \psi_h^{n-1/2}\right)
		\\
		& \textstyle  \leq  \varepsilon\left(\frac{E_u^n-E_u^{n-1}}{\tau},  \psi_h^{n-1/2}\right)
		+ \frac{\nu}{4} \Vert\nabla \overline{E}_u^{n-1/2}\Vert^2 
		+\nu \varepsilon^2 \Vert\nabla \psi_h^{n-1/2}\Vert^2
		+ \frac{\mu}{4} \Vert\nabla \cdot \overline{E}_u^{n-1/2}\Vert^2  \\
		& \quad \textstyle  + \mu \varepsilon^2 \Vert\nabla  \psi_h^{n-1/2}\Vert^2 
		+ \varepsilon c\left(\overline{E}_u^{n-1/2}, \overline{\bm{\xi}}_h^{n-1/2}, \psi_h^{n-1/2}\right) + \varepsilon c\left(\overline{u}_h^{n-1/2}, \overline{E}_u^{n-1/2}, \psi_h^{n-1/2}\right),
	\end{aligned}
\end{equation*}
and
\[
\begin{aligned}
	\textstyle  \varepsilon c\left(\overline{E}_u^{n-1/2}, \overline{\bm{\xi}}_h^{n-1/2}, \psi_h^{n-1/2}\right) 
	&\textstyle = \frac{\varepsilon}{2} \left((\overline{E}_u^{n-1/2} \cdot \nabla) \overline{\bm{\xi}}_h^{n-1/2}, \psi_h^{n-1/2}\right) - \frac{\varepsilon}{2} \left((\overline{E}_u^{n-1/2} \cdot \nabla) \psi_h^{n-1/2}, \overline{\bm{\xi}}_h^{n-1/2}\right) \\
	&\textstyle = \varepsilon \left((\overline{E}_u^{n-1/2} \cdot \nabla) \overline{\bm{\xi}}_h^{n-1/2}, \psi_h^{n-1/2}\right) + \frac{\varepsilon}{2} \left(\nabla \cdot \overline{E}_u^{n-1/2} \psi_h^{n-1/2}, \overline{\bm{\xi}}_h^{n-1/2}\right) \\
	&\textstyle \leq \varepsilon \Vert\overline{E}_u^{n-1/2}\Vert \cdot \Vert\nabla \overline{\bm{\xi}}_h^{n-1/2}\Vert_{L^\infty} \Vert\psi_h^{n-1/2}\Vert + \frac{\varepsilon}{2} \Vert\nabla \cdot \overline{E}_u^{n-1/2}\Vert \cdot \Vert\psi_h^{n-1/2}\Vert \cdot \Vert\overline{\bm{\xi}}_h^{n-1/2}\Vert_{L^\infty} \\
	&\textstyle \leq \frac14\Vert\overline{E}_u^{n-1/2}\Vert^2  + \frac{\mu}{4} 
\Vert	\nabla \cdot \overline{E}_u^{n-1/2}\Vert^2  + \varepsilon^2 C_5^2 C_{Pi}^2 (1+\frac{1}{4\mu})\Vert\nabla \psi_h^{n-1/2}\Vert^2.
\end{aligned}
\]
In addition,
$$\varepsilon c\left(\overline{\bu}_h^{n-1/2}, \overline{E}_u^{n-1/2}, \psi_h^{n-1/2}\right)=-\varepsilon c\left(\overline{E}_u^{n-1/2}, \overline{E}_u^{n-1/2}, \psi_h^{n-1/2}\right)+ \varepsilon c\left(\overline{\bm{\xi}}_h^{n-1/2}, \overline{E}_u^{n-1/2}, \psi_h^{n-1/2}\right),$$
where 
\[
\begin{aligned}
	\textstyle -\varepsilon c\left(\overline{E}_u^{n-1/2}, \overline{E}_u^{n-1/2}, \psi_h^{n-1/2}\right)
	&\textstyle =-\frac{\varepsilon}{2} \left((\overline{E}_u^{n-1/2} \cdot \nabla) \overline{E}_u^{n-1/2}, \psi_h^{n-1/2}\right) + \frac{\varepsilon}{2} \left((\overline{E}_u^{n-1/2} \cdot \nabla) \psi_h^{n-1/2}, \overline{E}_u^{n-1/2}\right) \\
	&\textstyle =  \frac{\varepsilon}{2} \left( \nabla\cdot \overline{E}_u^{n-1/2} \overline{E}_u^{n-1/2}, \psi_h^{n-1/2}\right) +  \varepsilon \left((\overline{E}_u^{n-1/2} \cdot \nabla) \psi_h^{n-1/2}, \overline{E}_u^{n-1/2}\right).
\end{aligned}
\]
Because
\begin{equation*}
	\Vert \bu_h^n - P_{st} \bu(t_n)\Vert_{L^\infty} \leq C_{\text{inv}}h^{-\frac{d}{2}}\Vert \bu_h^n - P_{st} \bu(t_n)\Vert \leq C_{\text{inv}}h^{-\frac{d}{2}} \mathcal{B}(T)(\tau^2 + h^2 + \epsilon),
\end{equation*}
so if $\varepsilon\leq h^{\frac{d}{2}}$, we have 
\[
\Vert \bu_h^n - P_{st} \bu(t_n)\Vert_{L^\infty} \leq  3C_{\text{inv}}\mathcal{B}(T).
\]
Therefore,
\begin{equation*}
	\Vert \bu_h^n \Vert_{L^\infty}\leq \Vert  P_{st} \bu(t_n)\Vert_{L^\infty}+3C_{\text{inv}}\mathcal{B}(T):=C_{6}, \text{ and } \Vert \overline{E}_u^{n-1/2} \Vert_{L^\infty}\leq \frac{C_5+C_{6}}{2}.
\end{equation*}
\begin{equation*}
\textstyle 	-\varepsilon c\left(\overline{E}_u^{n-1/2}, \overline{E}_u^{n-1/2}, \psi_h^{n-1/2}\right) \leq \frac{\mu}{4} \Vert \nabla\cdot \overline{E}_u^{n-1/2}\Vert^2 + \frac14\Vert \overline{E}_u^{n-1/2}\Vert^2
	+  \varepsilon^2\left(\frac{C_5+C_{6}}{2}\right)^2 \left( \frac{1}{4\mu}C_{Pi}^2 + 1\right)\Vert \nabla\psi_h^{n-1/2}\Vert^2,
\end{equation*}
and
\begin{equation*}
	\varepsilon c\left( \overline{\bm{\xi}}^{n-1/2}, \overline{E}_u^{n-1/2}, \psi_h^{n-1/2}\right) = - \varepsilon( (\overline{\bm{\xi}}^{n-1/2}\cdot\nabla)\psi_h^{n-1/2},  \overline{E}_u^{n-1/2})
	\leq \frac14 \Vert \overline{E}_u^{n-1/2} \Vert^2 + \varepsilon^2C_5^2\Vert \nabla \psi_h^{n-1/2}\Vert^2.
\end{equation*}
Combining all the above estimates with $\Vert\nabla \psi_h^{n-1/2}\Vert\leq C_{\text{inf}} \left\Vert\frac{r_h^n - r_h^{n-1}}{\tau} \right\Vert$, we have
\begin{equation*}
	\begin{aligned}
		&\textstyle	 \frac{\Vert E_u^n \Vert^2 -   \Vert E_u^{n-1} \Vert^2}{2\tau}   + \varepsilon  \frac{\Vert E_p^n \Vert^2 -   \Vert E_p^{n-1} \Vert^2}{2\tau}+ \frac{3\nu}{4}\Vert\nabla \overline{E}_u^{n-1/2}\Vert^2 \\
		&\textstyle 
		\leq    \varepsilon\left(\frac{E_u^n-E_u^{n-1}}{\tau},  \psi_h^{n-1/2}\right) + \left(\frac{C_{5}^2}{4\mu}+C_{5} + \frac34\right)\Vert \overline{E}_u^{n-1/2} \Vert ^2\\
		&\quad
		\textstyle +\varepsilon^2 \left( \nu +\mu + C_5^2 C_{Pi}^2 \left(1+\frac{1}{4\mu} \right) + \left(\frac{C_5+C_{6}}{2}\right)^2 \left( \frac{1}{4\mu}C_{Pi}^2 + 1\right) + C_5^2\right) C_{\text{inf}}^2 \left\Vert\frac{r_h^n - r_h^{n-1}}{\tau} \right\Vert^2.
	\end{aligned}
\end{equation*}
Changing the index $n$ to $i$, summing $i$ from 1 to $n$, it yields 
\begin{equation*}
	\begin{aligned}
		& \textstyle	 \frac{\Vert E_u^n \Vert^2 -   \Vert E_u^{0} \Vert^2}{2\tau}   + \varepsilon  \frac{\Vert E_p^n \Vert^2 -   \Vert E_p^{0} \Vert^2}{2\tau}+ \frac{3\nu}{4} \sum\limits_{i=1}^n\Vert\nabla \overline{E}_u^{i-1/2}\Vert^2 \\
		&\textstyle 
		\leq    \varepsilon \sum\limits_{i=1}^n \left(\frac{E_u^i-E_u^{i-1}}{\tau},  \psi_h^{i-1/2}\right) + \left(\frac{C_{5}^2}{4\mu}+C_{5} + \frac34\right) \sum\limits_{i=1}^n\Vert \overline{E}_u^{i-1/2} \Vert ^2\\
		&\quad
		\textstyle +\varepsilon^2 \left( \nu +\mu + C_5^2 C_{Pi}^2 \left(1+\frac{1}{4\mu} \right) + \left(\frac{C_5+C_{6}}{2}\right)^2 \left( \frac{1}{4\mu}C_{Pi}^2 + 1\right) + C_5^2\right) C_{\text{inf}}^2 \sum\limits_{i=1}^n \left\Vert\frac{r_h^i - r_h^{i-1}}{\tau} \right\Vert^2,
	\end{aligned}
\end{equation*}
where
\begin{equation*}
	\begin{aligned}
\textstyle 	\varepsilon\sum\limits_{i=1}^n \left(\frac{E_u^i-E_u^{i-1}}{\tau},  \psi_h^{i-1/2}\right)  &\textstyle =  \frac{\varepsilon}{\tau}\left(E_u^n, \psi_h^{n-1/2}\right) - \varepsilon \sum\limits_{i=1}^{n-1} \left( E_u^i, \frac{\psi_h^{i+1/2} - \psi_h^{i-1/2}}{\tau}\right)\\
&\textstyle \leq \frac{1}{4\tau}\Vert E_u^n \Vert^2 + \frac{\varepsilon^2}{\tau} C_{\text{inf}}^2 C_{\text{Pi}}^2 \Vert \frac{r_h^n - r_h^{n-1}}{\tau} \Vert^2
+ \frac14 \sum\limits_{i=1}^{n-1} \Vert E_u^i \Vert^2 
+  \varepsilon^2C_{\text{inf}}^2 C_{\text{Pi}}^2 \sum\limits_{i=1}^{n-1} \left\Vert  \frac{ r_h^{i+1}- 2r_h^i +   r_h^{i-1} }{\tau^2}\right \Vert^2.
\end{aligned}
\end{equation*}
Hence, we can further have
\begin{equation}\label{app:roug}
	\begin{aligned}
		& \textstyle	 \Vert E_u^n \Vert^2 -   \Vert E_u^{0} \Vert^2   + \varepsilon \Vert E_p^n \Vert^2 -  \varepsilon \Vert E_p^{0} \Vert^2+  \nu \sum\limits_{i=1}^n \tau \Vert\nabla \overline{E}_u^{i-1/2}\Vert^2 \\
		&\textstyle 
		\leq   \frac{1}{2}\Vert E_u^n \Vert^2 + 2\varepsilon^2  C_{\text{inf}}^2 C_{\text{Pi}}^2 \Vert \frac{r_h^n - r_h^{n-1}}{\tau} \Vert^2
		+  2\varepsilon^2C_{\text{inf}}^2 C_{\text{Pi}}^2 \sum\limits_{i=1}^{n-1} \tau \left\Vert  \frac{ r_h^{i+1}- 2r_h^i +   r_h^{i-1} }{\tau^2}\right \Vert^2  + \left(\frac{C_{5}^2}{4\mu}+ C_{5} + \frac54 \right) \sum\limits_{i=1}^n \tau \Vert E_u^i \Vert ^2\\
		&\quad
		\textstyle +2\varepsilon^2 \left( \nu +\mu + C_5^2 C_{Pi}^2 \left(1+\frac{1}{4\mu} \right) + \left(\frac{C_5+C_{6}}{2}\right)^2 \left( \frac{1}{4\mu}C_{Pi}^2 + 1\right) + C_5^2\right) C_{\text{inf}}^2 \sum\limits_{i=1}^n \tau \left\Vert\frac{r_h^i - r_h^{i-1}}{\tau} \right\Vert^2.
	\end{aligned}
\end{equation}
Now we turn to show the boundedness of $\textstyle \left\Vert \frac{r_h^n - r_h^{n-1}}{\tau} \right\Vert^2$, $\textstyle \sum\limits_{i=1}^n \tau\left\Vert \frac{r_h^i - r_h^{i-1}}{\tau} \right\Vert^2$, and $\textstyle \sum\limits_{i=1}^n \tau\left\Vert  \frac{ r_h^{i+1}- 2r_h^{i} +   r_h^{i-1} }{\tau^2}\right \Vert^2$. Given $\tau\leq h^{\frac12+\frac{d}{4}}$ and Lemma \ref{op:reg:ns}, we have
\begin{equation*}
	\begin{aligned}
		\textstyle \left\Vert \frac{r_h^n - r_h^{n-1}}{\tau} \right\Vert^2 \leq&  \textstyle~
		3  \Vert  \partial_\tau E_r^n \Vert^2 +  3  \Vert  \partial_\tau p(t_n) - {\rm Q}_{\text{st}}p(t_n)  \Vert^2 +  3   \Vert    \partial_\tau p(t_n)  \Vert^2
		\leq   \widetilde{C}\left( \frac{\tau^4+h^4}{\tau^3} +  \frac{h^4}{\tau^2}+ 1 \right) \leq \widetilde{C},
	\end{aligned}
\end{equation*}
\begin{equation*}
\begin{aligned}
	\textstyle \sum\limits_{i=1}^n \tau\left\Vert \frac{r_h^i - r_h^{i-1}}{\tau} \right\Vert^2 \leq&  
	\textstyle 3\sum\limits_{i=1}^n \tau \Vert  \partial_\tau E_r^i \Vert^2 +  3 \sum\limits_{i=1}^n \tau\Vert  \partial_\tau p(t_i) - {\rm Q}_{\text{st}}p(t_i)  \Vert^2 +  3  \sum\limits_{i=1}^n \tau\Vert    \partial_\tau p(t_i)  \Vert^2 \leq \widetilde{C} (\frac{\tau^4+h^4}{\tau^2} + \frac{h^4}{\tau^2}+1)\leq \widetilde{C},
\end{aligned}
\end{equation*}
and,
\begin{equation*}
	\begin{aligned}
		\textstyle \sum\limits_{i=1}^n \tau \left\Vert  \frac{ r_h^{i+1}- 2r_h^{i} +   r_h^{i-1} }{\tau^2}\right \Vert^2 \leq&~  \textstyle 3\sum\limits_{i=1}^n \tau \Vert \partial_\tau^2 E_r^{i+1} \Vert^2 + 3 \sum\limits_{i=1}^n \tau\Vert \partial_\tau^2 p(t_{i+1})  - {\rm Q}_{\text{st}}\partial_\tau^2 p(t_{i+1}) \Vert^2 + 3\sum\limits_{i=1}^n \tau\Vert  \partial_\tau^2 p(t_{i+1})  \Vert^2\\
		\leq& \widetilde{C} \left(\frac{\tau^4+h^4}{\tau^4} + \frac{h^4}{\tau^4} + 1 \right)\leq \widetilde{C}.
	\end{aligned}
\end{equation*}
Thereafter, \eqref{app:roug} can be simplifed as
\begin{equation}\label{app:final}
	\begin{aligned}
		& \textstyle\frac12 \Vert E_u^n \Vert^2   + \varepsilon \Vert E_p^n \Vert^2  +  \nu \sum\limits_{i=1}^n \tau \Vert\nabla \overline{E}_u^{i-1/2}\Vert^2 
		\leq   
		 \left(\frac{C_{5}^2}{4\mu}+ C_{5} + \frac54 \right) \sum\limits_{i=1}^n \tau \Vert E_u^i \Vert ^2 + \widetilde{C} \varepsilon^2.
	\end{aligned}
\end{equation}
Lastly, when $\tau\leq \frac14 \left(\frac{C_{5}^2}{4\mu}+ C_{5} + \frac54 \right)^{-1}$, by Gr\"onwall's inequality, we can finally have the following approximation
\begin{equation*}
	\textstyle\frac12 \Vert E_u^n \Vert^2   + \varepsilon \Vert E_p^n \Vert^2  +  \nu \sum\limits_{i=1}^n \tau \Vert\nabla \overline{E}_u^{i-1/2}\Vert^2 \leq  \widetilde{C} \varepsilon^2.
\end{equation*}

\end{proof}

Based on Lemma \ref{op:reg:ns} and Theorem \ref{approximation}, we can directly have the following corollary via the triangle inequality.
\begin{cor}\label{corolla}
Given $h\leq 1$, if $\textstyle \tau\leq \min\{h^{\frac12 + \frac{d}{4}}, \frac14 \left(\frac{C_{5}^2}{4\mu}+ C_{5} + \frac54 \right)^{-1}\}$, we can have the following estimate
\begin{equation}
	\textstyle	 \nu \tau\sum\limits_{i=1}^n \Vert \nabla\overline{e}_u^{i-1/2} \Vert^2 + \mu  \tau\sum\limits_{i=1}^n \Vert \nabla\cdot \overline{e}_u^{i-1/2} \Vert^2   \leq  \widetilde{C} \left( \varepsilon^2 +\tau^4+h^4\right),\text{ and }
	\textstyle	 \Vert e_u^n \Vert^2  \leq  \widetilde{C} \left( \varepsilon^2 +\tau^4+h^6\right).
\end{equation}	
\end{cor}

\begin{remark}
	Corollary \ref{corolla} provides an alternative approach to obtaining the Reynolds number robust error estimate for the CN-ACM scheme \eqref{acm:gdiv:eq}. However, this result relies on the mesh constraint $ \tau\leq h^{\frac12 + \frac{d}{4}}$. Although Theorem \ref{op:cn:h1} does not yield the optimal $L^2$ error estimate, it establishes Reynolds number robustness independent of the mesh ratio.
\end{remark}
\section{Numerical Experiments}\label{NumExp}

\subsection{Convergence tests}
We investigate the convergence behavior of the CN-ACM scheme. The viscosity is taken as $\nu = 0.01$, and the stabilization parameter is set to $\mu = 1$. For the spatial discretization, we utilize the $P_2$-$P_1^{\text{disc}}$ finite element pair on meshes generated by Alfeld splitting (as shown in Figure \ref{fig:meshes}), with a mesh size of $h = 1/n$ and a time step of $\tau = T/N_T$. All numerical experiments are implemented using FreeFem++ \cite{hecht2012new}.

\begin{figure}[htbp]
	\centering
	\includegraphics[width=0.45\textwidth]{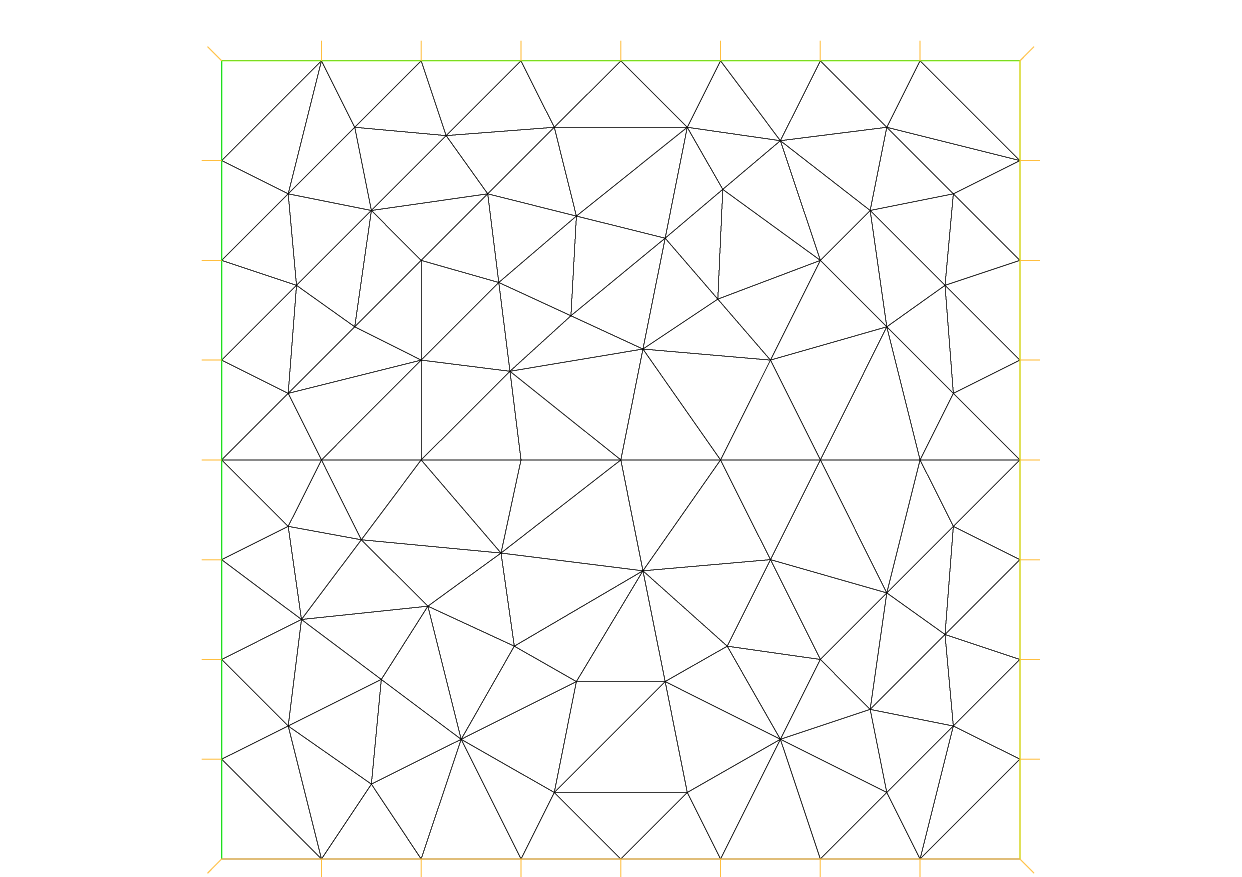}
	\includegraphics[width=0.45\textwidth]{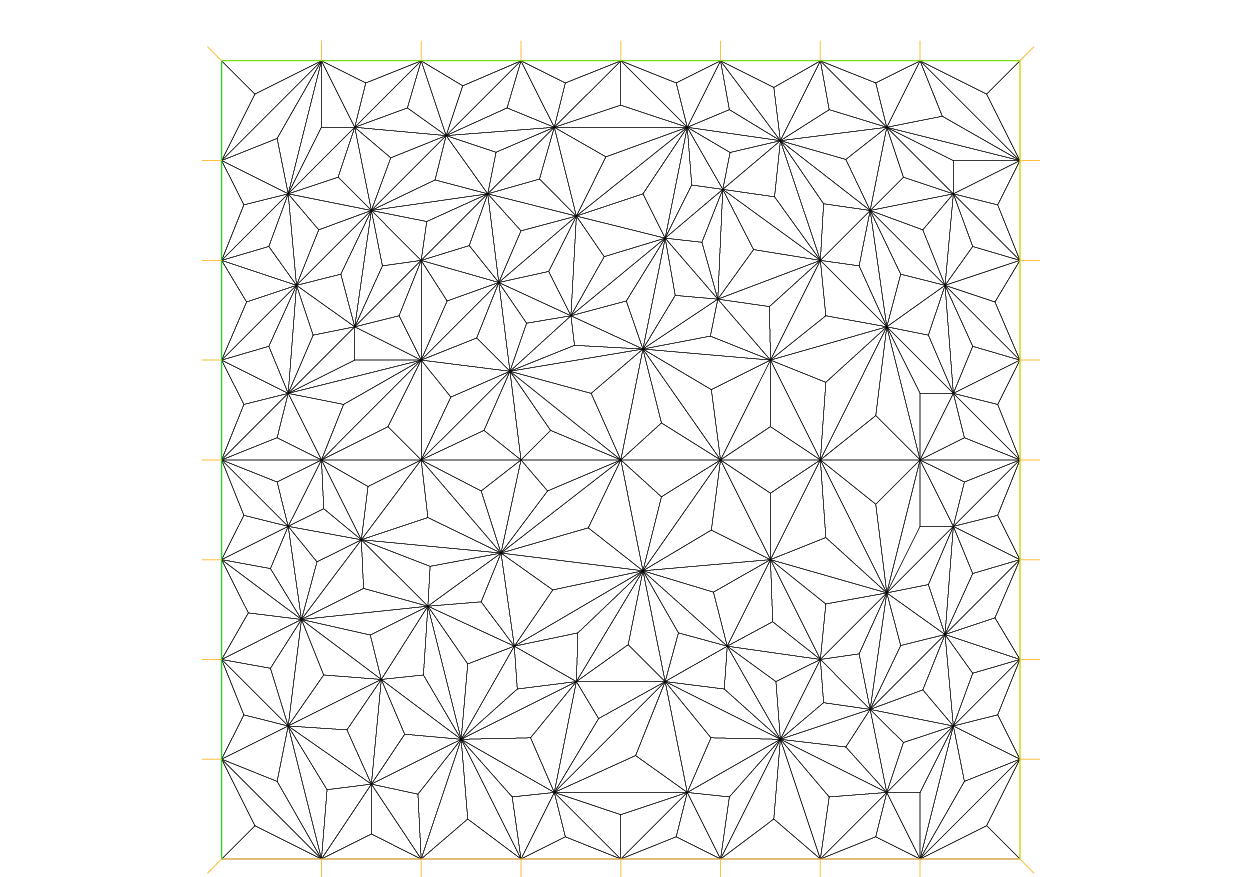}
	\caption{Initial Delaunay mesh (left) and its Alfeld splitting (right).}% and Powell-Sabin splitting (right)
	\label{fig:meshes}
\end{figure}

The numerical errors are measured using the $L^2$-norm and $H^1$-seminorm for the velocity, and the $L^2$-norm for the pressure. To validate the convergence rates, we adopt the following manufactured exact solutions:
\begin{align*}
	u(x,y,t) &= \sin(2\pi x)\sin(2\pi y)\, e^{-8\nu\pi^2 t},\\
	v(x,y,t) &= \cos(2\pi x)\cos(2\pi y)\, e^{-8\nu\pi^2 t},\\
	p(x,y,t) &= -\frac{1}{2}
	\left(\sin^2(2\pi x) + \cos^2(2\pi y)\right)
	e^{-16\nu\pi^2 t}.
\end{align*}
The corresponding external forcing term is identical to zero, i.e.,
$
\bm{f} = (f_1,f_2)^{T} = \mathbf{0}.
$

We first test the convergence rate with respect to $\varepsilon$. To this end, we fix $n=80$ and $N_T=80$, and gradually decrease $\varepsilon$ by a factor of 2 starting from $0.2$. The results are presented in Table \ref{tab:error_rate2}, which demonstrates a first-order convergence rate for both the $L^2$ and $H^1$ norms of the velocity, as well as the $L^2$ norm of the pressure. This implies that the estimates in Theorem \ref{op:cn:h1} and Corollary \ref{corolla} are optimal with respect to $\varepsilon$, and that the error constants do not contain the factor $1/\varepsilon$.

Furthermore, we evaluate the spatial and temporal convergence rates by simultaneously refining the mesh and time step. Starting from $n=8$ and $N_T=40$, both $n$ and $N_T$ are doubled at each step, with the parameter set to $\varepsilon = \tau^2$. The corresponding results are presented in Table \ref{tab:CN_convergence}. As the mesh is refined, the velocity error in the $L^2$-norm exhibits a convergence rate of $\mathcal{O}(h^3)$, while the velocity error in the $H^1$-seminorm and the pressure error in the $L^2$-norm both demonstrate optimal convergence rates of $\mathcal{O}(h^2)$. These findings strongly validate our theoretical predictions. 
\begin{table}[htbp]
	\centering
	\caption{Convergence results with $\nu=0.01$, $\mu=1$, $n=80$, $N_T=80$, and different $\varepsilon$.} 
	\label{tab:error_rate2} 
	\begin{tabular}{ccccccc}
		\hline
		$\varepsilon$ & $\|u-u_h\|_{L^2}$ & rate & $|u-u_h|_{H^1}$ & rate & $\|p-p_h\|_{L^2}$ & rate \\
		\hline
		0.2   & $4.664\mathrm{e}{-3}$ & --   & $5.938\mathrm{e}{-2}$ & --   & $2.446\mathrm{e}{-2}$ & --   \\ 
		0.1   & $2.374\mathrm{e}{-3}$ & 0.97 & $2.927\mathrm{e}{-2}$ & 1.02 & $1.036\mathrm{e}{-2}$ & 1.24 \\ 
		0.05  & $1.189\mathrm{e}{-3}$ & 1.00 & $1.461\mathrm{e}{-2}$ & 1.00 & $4.749\mathrm{e}{-3}$ & 1.13 \\ 
		0.025 & $5.940\mathrm{e}{-4}$ & 1.00 & $7.674\mathrm{e}{-3}$ & 0.93 & $2.278\mathrm{e}{-3}$ & 1.06 \\ 
		\hline
	\end{tabular}
\end{table}

\begin{table}[htbp]
	\centering
	\caption{Convergence results with $\nu=0.01$, $\mu=1$, $N_T=5n$, and $\varepsilon=\tau^2$.}
	\label{tab:CN_convergence} 
	\begin{tabular}{cccccccc}
		\toprule
		$n$ & $N_T$ & $\|u-u_h\|_{L^2}$ & rate & $|u-u_h|_{H^1}$ & rate & $\|p-p_h\|_{L^2}$ & rate \\ 
		\midrule
		8  & 40  & $3.082\mathrm{e}{-3}$ & --   & $2.329\mathrm{e}{-1}$ & --   & $3.353\mathrm{e}{-2}$ & --   \\
		16 & 80  & $3.615\mathrm{e}{-4}$ & 3.09 & $5.390\mathrm{e}{-2}$ & 2.11 & $8.582\mathrm{e}{-3}$ & 1.97 \\
		32 & 160 & $4.082\mathrm{e}{-5}$ & 3.15 & $1.207\mathrm{e}{-2}$ & 2.16 & $2.028\mathrm{e}{-3}$ & 2.08 \\
		64 & 320 & $5.437\mathrm{e}{-6}$ & 2.91 & $3.225\mathrm{e}{-3}$ & 1.90 & $5.268\mathrm{e}{-4}$ & 1.95 \\
		\bottomrule
	\end{tabular}
\end{table}

\subsection{Gresho vortex standing problem}

We next consider the Gresho standing vortex problem \cite{gresho1990theory}.
The domain $\Omega = (-0.5, 0.5)\times (-0.5, 0.5)$. The initial condition is given by the  solution of an steady Euler equation:
\begin{equation}
	\begin{aligned}
		r \leq 0.2:&
		\left\{
		\begin{array}{l}
			\bm u_0=\binom{-5 y}{5 x} \\
			p_0=12.5 r^2+K_1
		\end{array},\right. 
		\\
		r>0.4:&\left\{\begin{array}{l}
			\bm u_0=\binom{0}{0} \\
			p_0=0
		\end{array},\right. 
		\\
		0.2 \leq r \leq 0.4:&\left\{\begin{array}{l}
			\bm u_0=\left(\begin{array}{l}
				\frac{-2 y}{r}+5 y\\
				\frac{2x}{r} - 5 x
			\end{array}
			\right) \\
			p_0=12.5 r^2-20 r+4 \log (r)+K_2
		\end{array},\right.
	\end{aligned}
\end{equation}
where
$$
r = \sqrt{x^2+y^2},\quad K_2=(-12.5)(0.4)^2+20(0.4)^2-4 \log (0.4),\quad K_1=C_2-20(0.2)+4 \log (0.2) .
$$
The external force $\bm f = 0$, and the zero-Dirichlet boundary conditions is imposed. The viscosity number is set as $\nu=0.0$. The terminal time is chosen to be $T=5$. In this problem, the initial solution configuration will be well maintained along the time, with almost neglectible diffusion caused by the tiny viscosity. We first generate the Delaunay mesh with $h={1}/{60}$, then apply Alfeld splitting. The time step size is $\tau=0.005$. The corresponding snapshots are shown in Figure \ref{figure:gresho}. In addition, we also compare the evolutions of $\Vert \nabla\cdot \bm u\Vert$ for $\mu=1$ and $\mu=0$, which is present in Figure \ref{figure:gresho:divl2}. We can see that $\Vert \nabla\cdot \bm u\Vert$ with $\mu=1$ decreases along with time, however, it does not share such property when $\mu=0$. 

\begin{figure}[htbp]	
	\centering
	\includegraphics[width=0.3\textwidth]{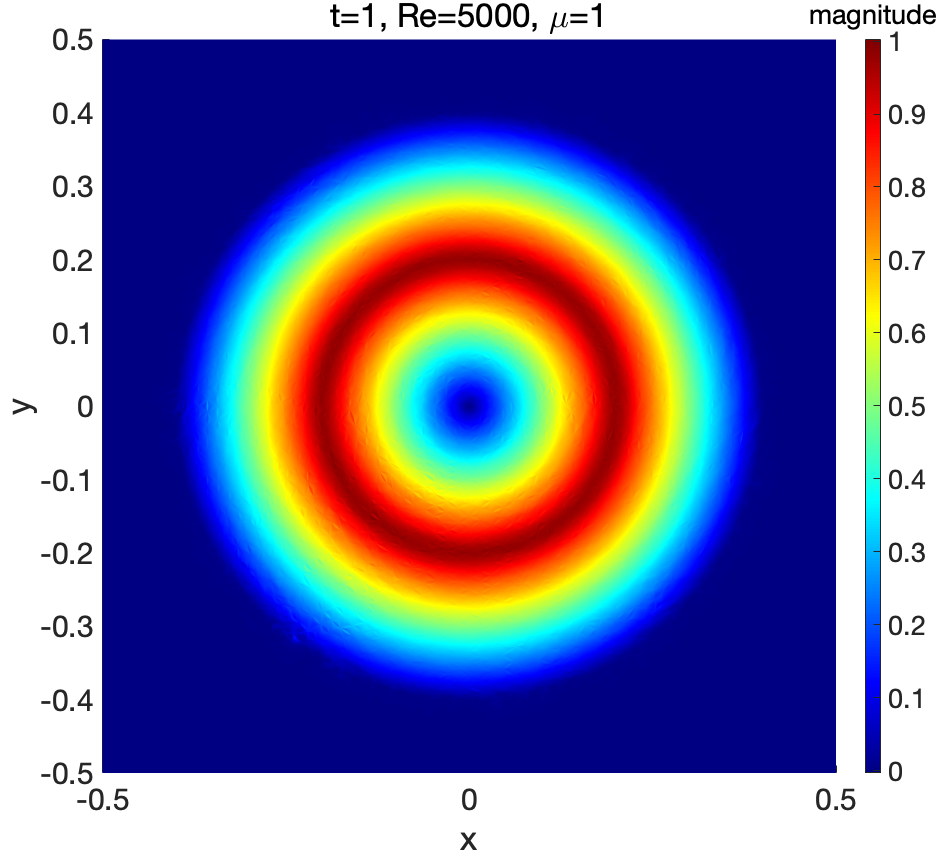}\hspace{4ex}
	\includegraphics[width=0.3\textwidth]{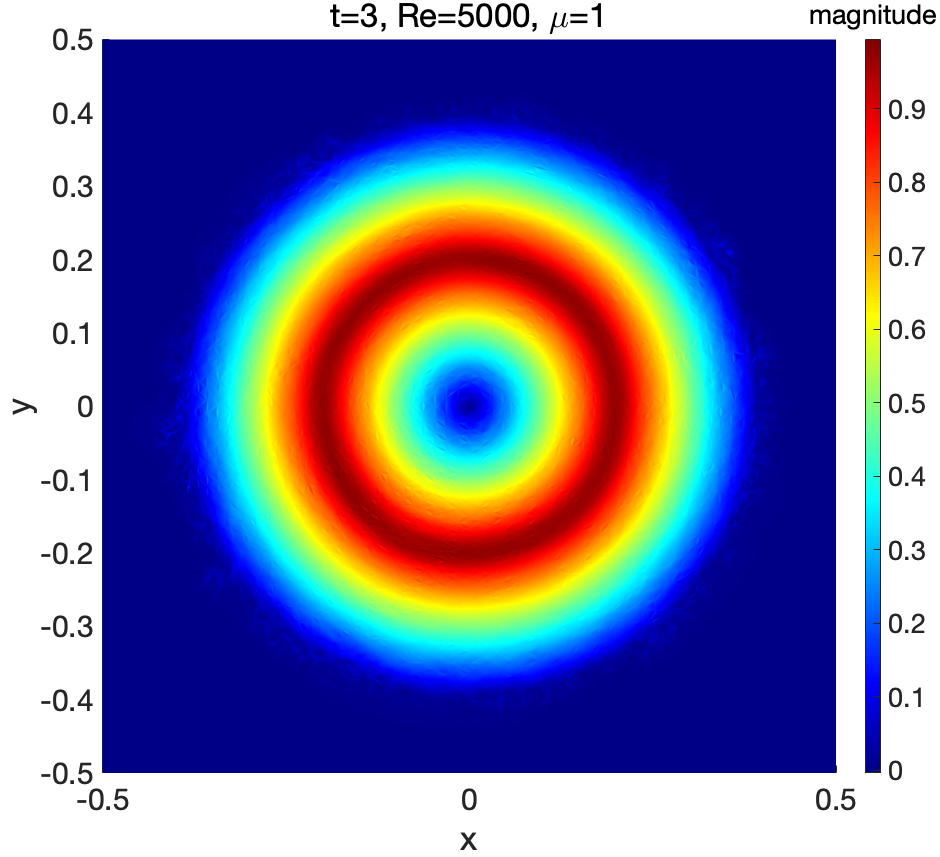}\hspace{4ex}
	\includegraphics[width=0.3\textwidth]{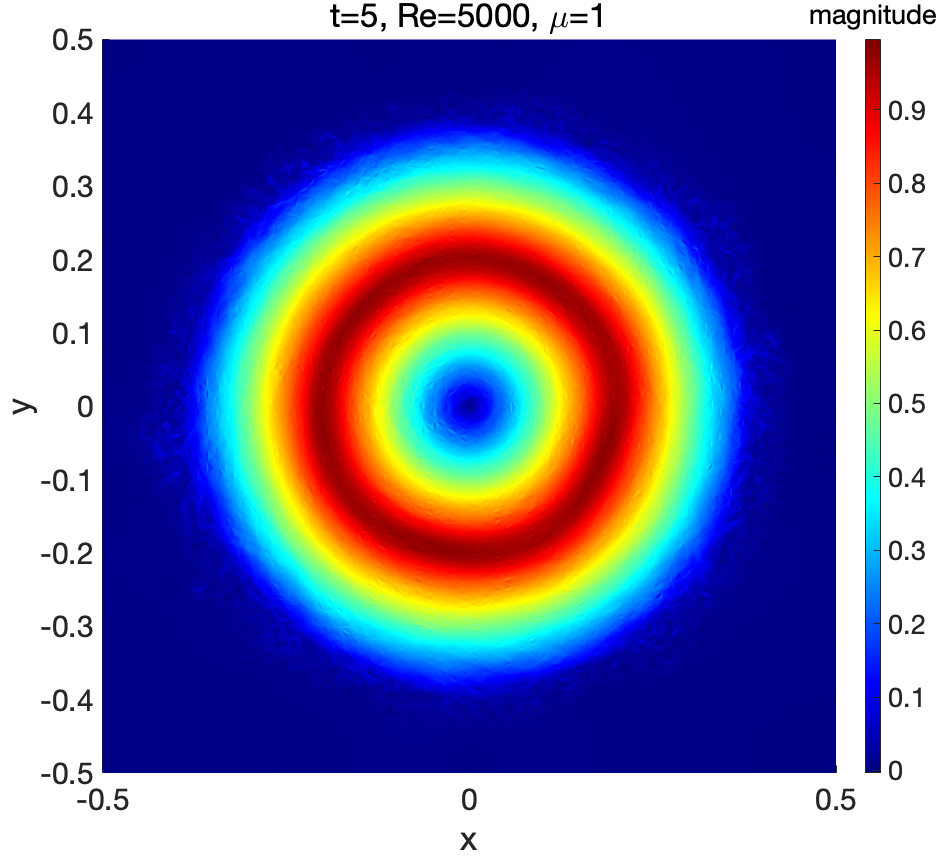}
	\caption{Time evolution of the velocity magnitude for the Gresho vortex standing problem at time $t=1, 3, \text{ and }5$.}
	\label{figure:gresho} 
\end{figure}

\begin{figure}[htbp]	
	\centering
	\includegraphics[width=0.35\textwidth]{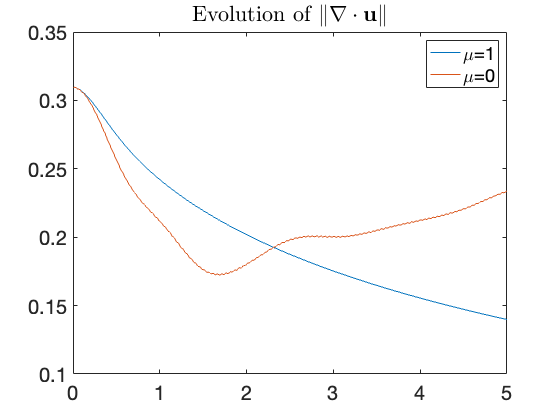}
	\caption{Evolution of $\Vert \nabla\cdot\bu \Vert$ for the Gresho vortex standing problem with $\mu=1$ and $\mu=0$.}
	\label{figure:gresho:divl2} 
\end{figure}

\subsection{Lid-driven cavity flow}

Next, we evaluate the robustness of the proposed CN-ACM scheme using the challenging lid-driven cavity flow benchmark \cite{1982High,JuWang2017}. For the 2D case, the computational domain is defined as $\Omega=(0,1)^2$. No-slip boundary conditions are imposed on the three stationary walls ($x = 0$, $x = 1$, and $y = 0$), while the top lid ($y = 1$) is driven by a tangential unit velocity. The kinematic viscosity is set to $\nu=2 \times 10^{-4}$ (corresponding to a Reynolds number of $Re=5000$), and the stabilization parameter is $\mu=1$. The computational mesh is constructed by applying Alfeld splitting to an initial Delaunay triangulation with $h=1/50$. We set the time step to $\tau=10^{-3}$ and the terminal time to $T=100$. Figure \ref{lid:evol} displays the contour plots of the velocity magnitude at $t=10, 20, 30, 40, 50$, and $100$, demonstrating the stability and robustness of the method.
Furthermore, Figure \ref{fig-mag-cont:mid} compares the velocity profiles along the centerlines at $t=100$ with the benchmark results from \cite{1982High}. The computed velocities exhibit excellent agreement with the reference data, particularly in capturing the local extrema and steep gradients, such as the $x$-component near $y=0.1$ and $y=0.97$ and the $y$-component near $x=0.1$ and $x=0.96$. These results confirm that the proposed CN-ACM scheme accurately captures the dynamical evolution of the velocity field.

\begin{figure}[htbp]
\centering
\begin{minipage}{0.32\textwidth}
\centering
\includegraphics[width=\textwidth]{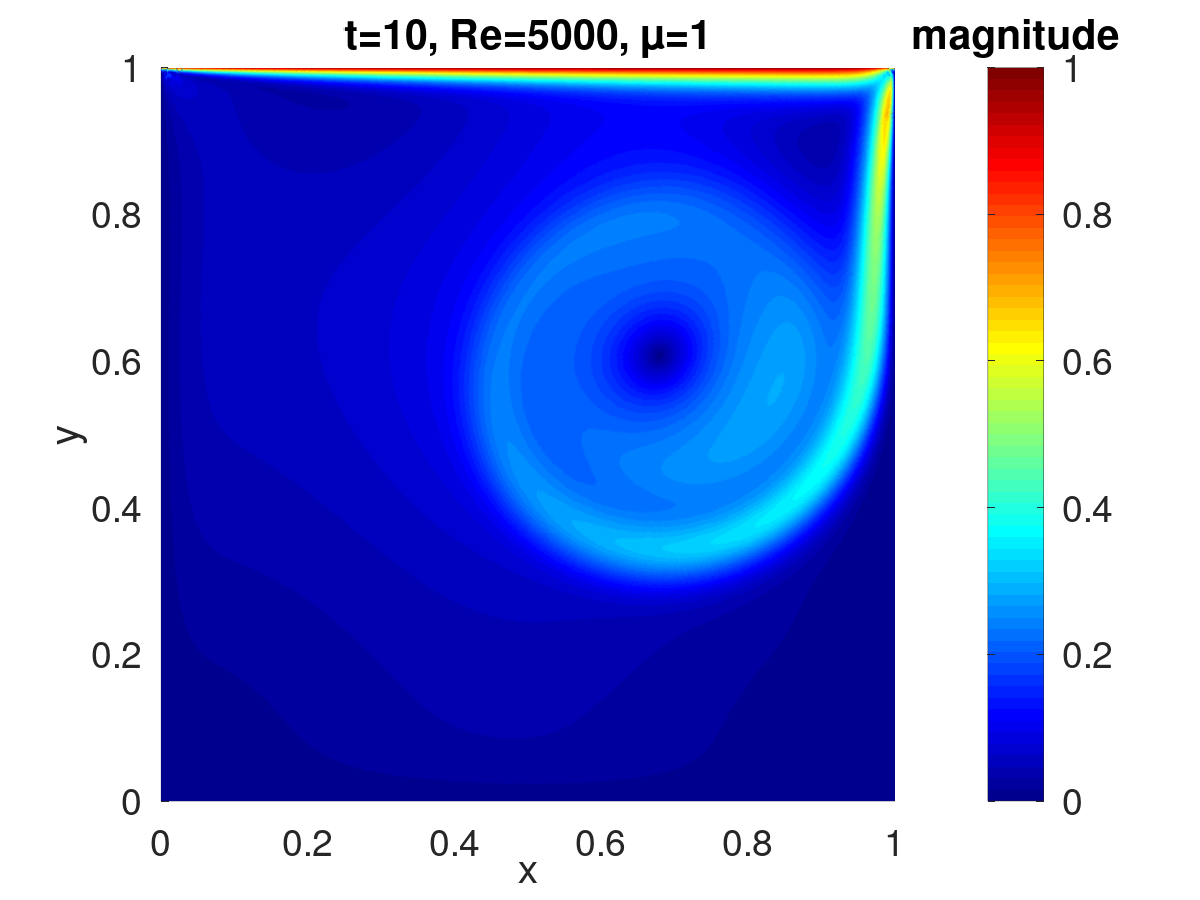}
\\ $(a)\ t=10$
\end{minipage}
\hfill
\begin{minipage}{0.32\textwidth}
\centering
\includegraphics[width=\textwidth]{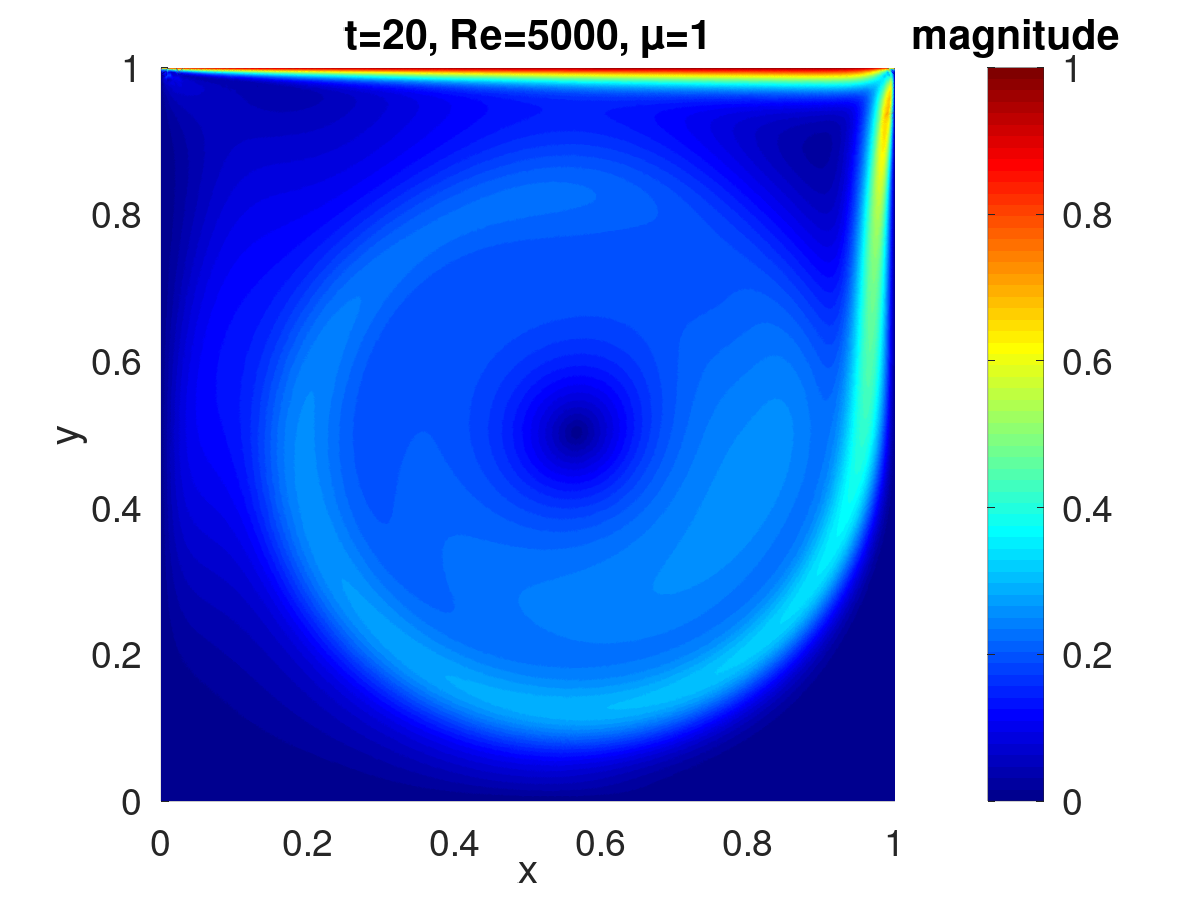}
\\ $(b)\ t=20$
\end{minipage}
\hfill
\begin{minipage}{0.32\textwidth}
\centering
\includegraphics[width=\textwidth]{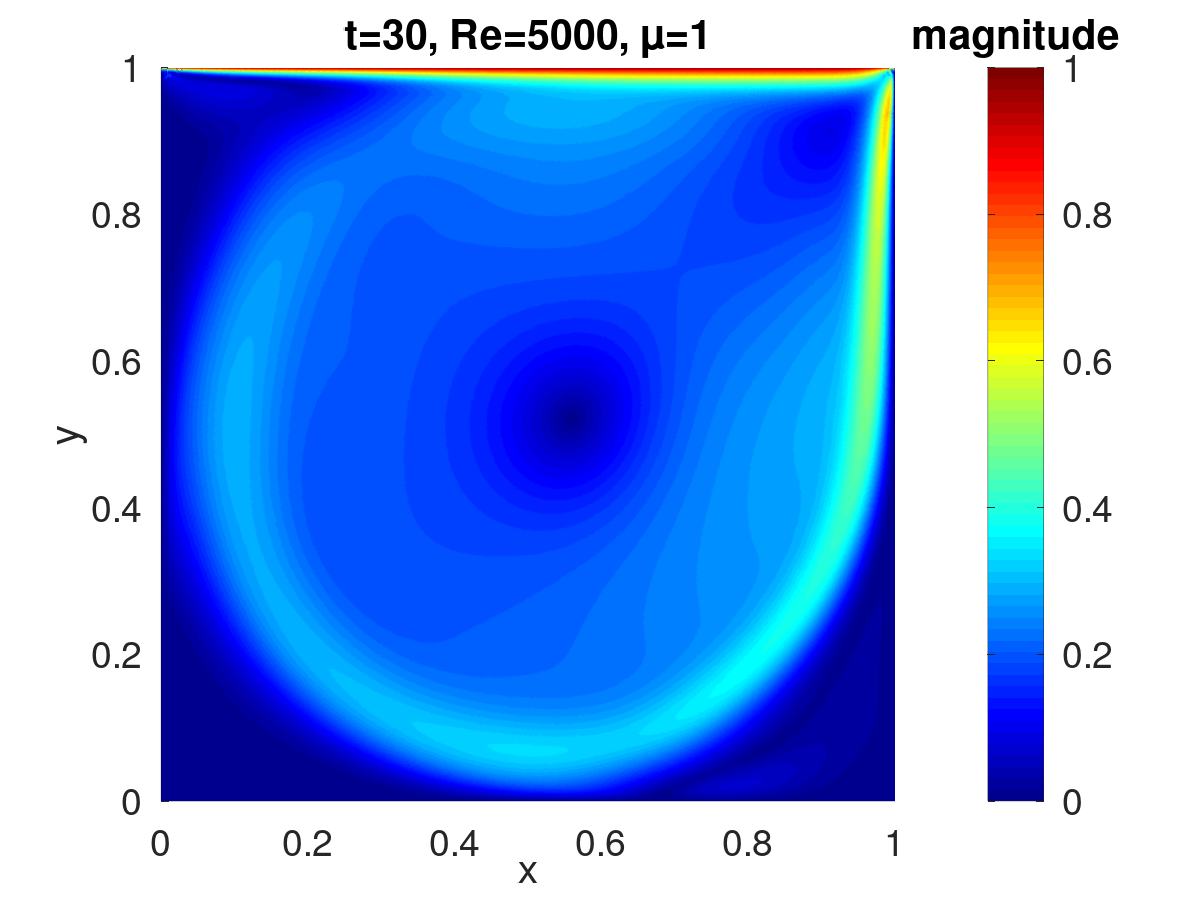}
\\ $(c)\ t=30$
\end{minipage}

\vspace{0.3cm}

\begin{minipage}{0.32\textwidth}
\centering
\includegraphics[width=\textwidth]{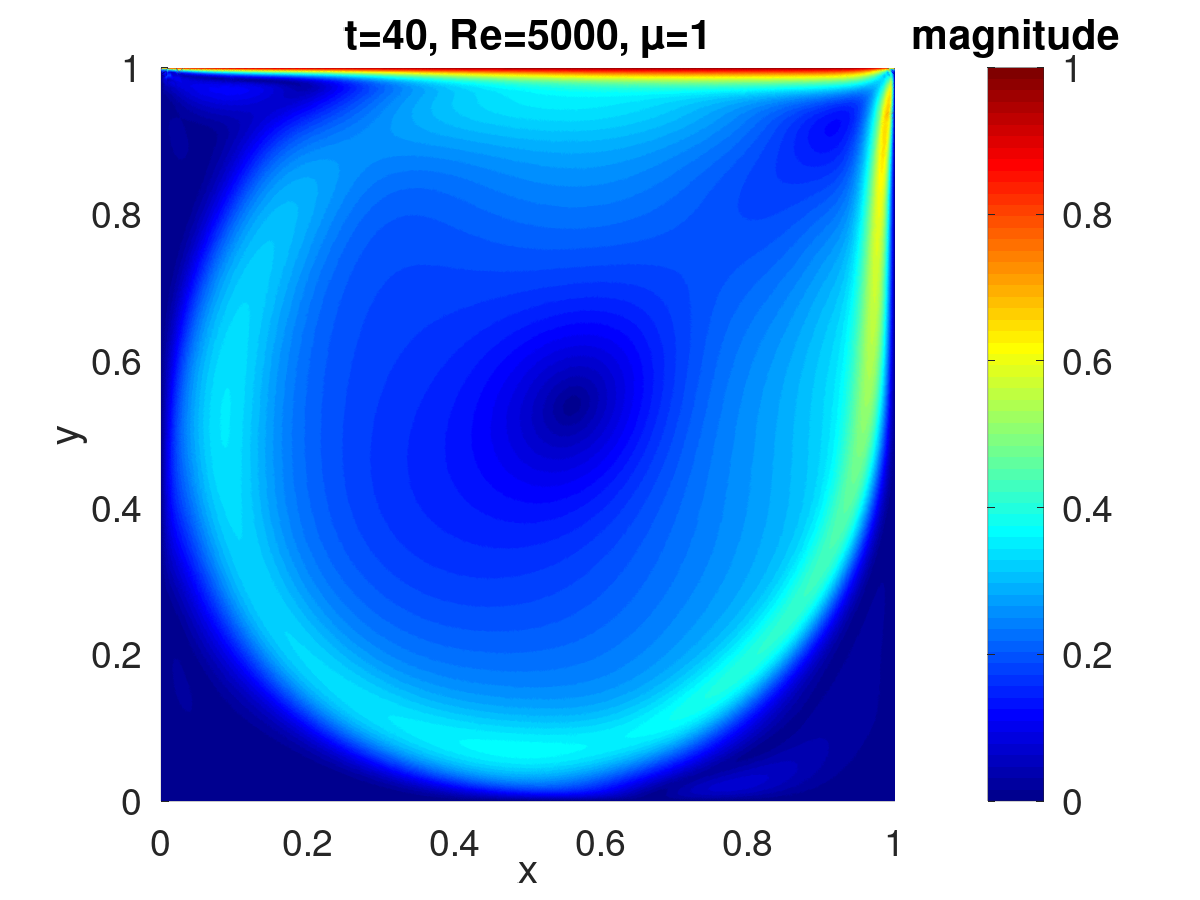}
\\ $(d)\ t=40$
\end{minipage}
\hfill
\begin{minipage}{0.32\textwidth}
\centering
\includegraphics[width=\textwidth]{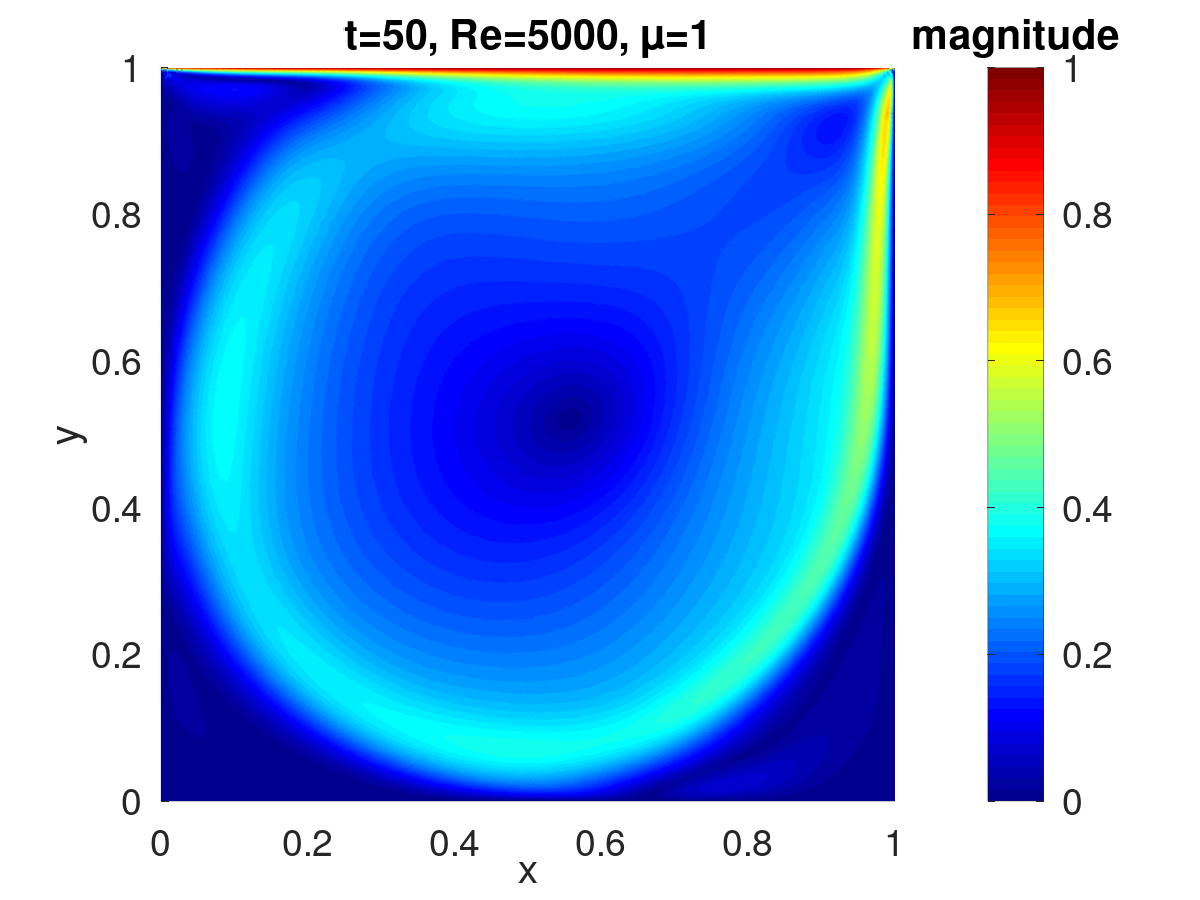}
\\ $(e)\ t=50$
\end{minipage}
\hfill
\begin{minipage}{0.32\textwidth}
\centering
\includegraphics[width=\textwidth]{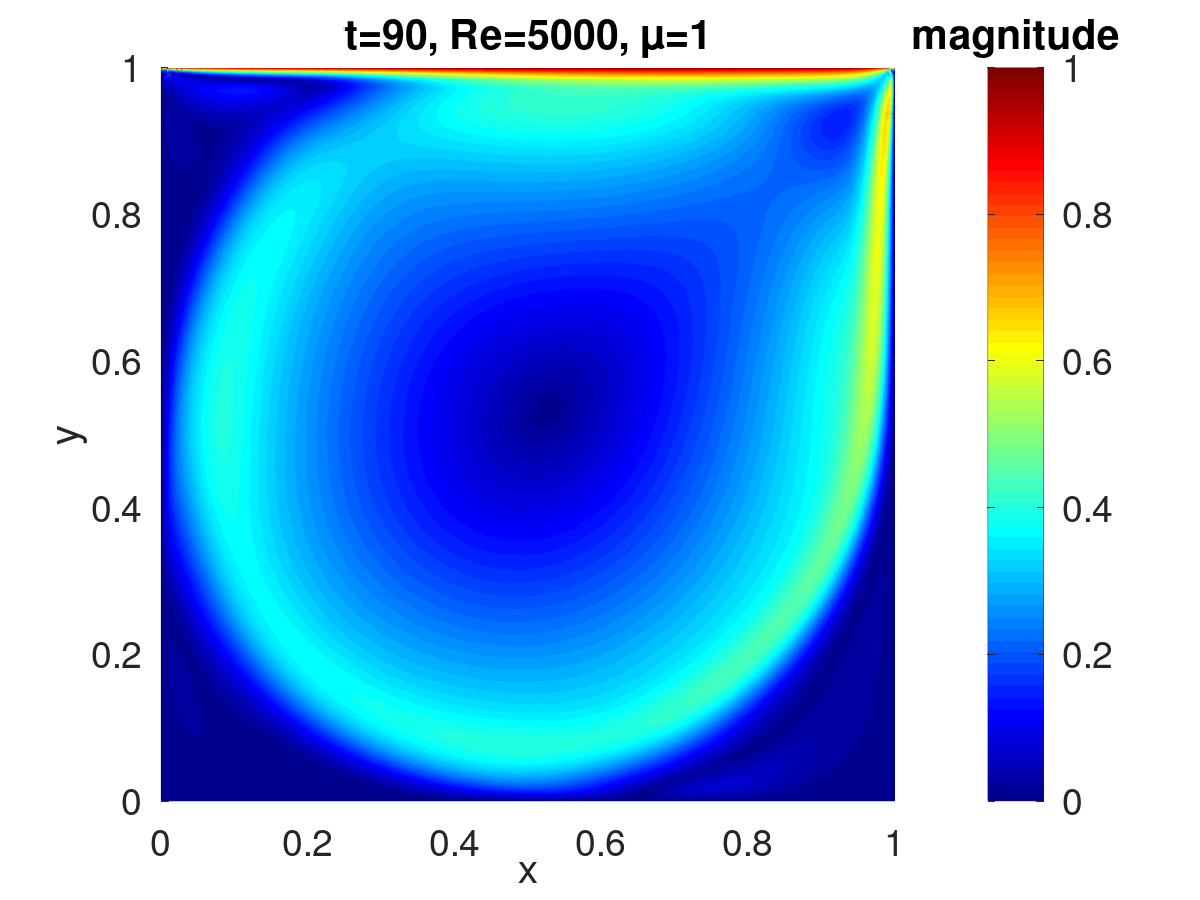}
\\ $(c)\ t=100$
\end{minipage}
\caption{Time evolution of the velocity magnitude for the lid-driven cavity flow.}
\label{lid:evol}
\end{figure}

\begin{figure}[htbp]
	\centering
	\includegraphics[width=0.35\textwidth]{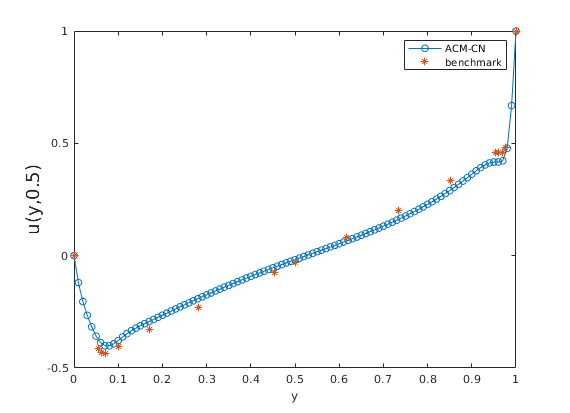}
	\includegraphics[width=0.35\textwidth]{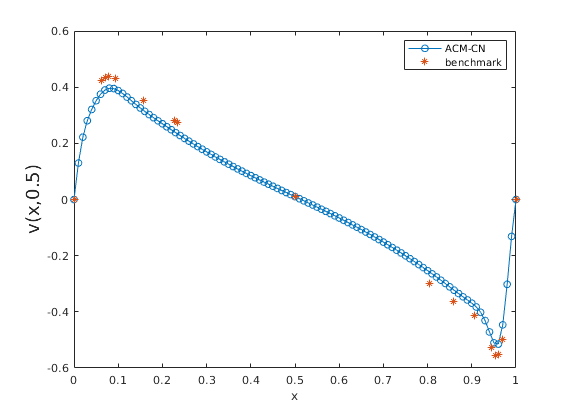} 
	\caption{The velocity at the center line with $x$-component velocity at $x = 0.5$ (left) and $y$-component velocity at $y = 0.5$ (right)  at $t=100$.}
	\label{fig-mag-cont:mid}
\end{figure}

\section{Conclusions}\label{conclusion}
In this paper, we have proposed and analyzed a second-order Crank-Nicolson fully discrete scheme based on the artificial compressibility method  for the incompressible Navier-Stokes equations. To address the instabilities arising from the artificial compressibility perturbation and high Reynolds numbers, we introduced the grad-div stabilization mechanism, coupled with the Scott-Vogelius finite element pair on Alfeld refined meshes. This specific combination not only facilitates the exact decoupling of velocity and pressure solves for computational efficiency but also lays the foundation for rigorous parameter-robust analysis.

The primary contribution of this work is the establishment of fully discrete optimal error estimates that are simultaneously robust with respect to both the viscosity $\varepsilon$ (inversely proportional to the Reynolds number) and the artificial compressibility parameter $\varepsilon$. Specifically, we have derived optimal convergence rates for the velocity in both the $L^2$ and $H^1$ norms, and for the pressure in the $L^2$ norm. A key feature of our analysis is the successful elimination of the $\varepsilon^{-1}$ dependency—which typically causes order reduction in standard ACM analyses—and the $\nu^{-1}$ dependency, which severely limits long-time accuracy. Furthermore, by circumventing the exponential dependence on the Reynolds number typically introduced by Grönwall’s lemma, our parameter-uniform bounds ensure the long-time accuracy and stability of the scheme.

The methodological and analytical frameworks developed in this work have several promising directions for future research. A natural and challenging extension is the application of the ACM strategy to fluid-coupled multiphysics models, such as thermally driven flows governed by the Boussinesq equations or incompressible magnetohydrodynamic (MHD) systems. These coupled systems inherently introduce additional physical parameters (e.g., thermal diffusivity, magnetic Reynolds number) and tighter coupling between multiple fields, which further complicate the numerical stability and error analysis. Future work will focus on designing efficient, decoupled ACM schemes for such multiphysics models and establishing corresponding fully discrete, parameter-robust optimal error estimates across all coupled physical variables.

\noindent {\bf Acknowledgements.} 
L. Ju's research was partially supported by U.S. National Science Foundation under grant numbers DMS-2109633 and DMS-2409634. R. Lan's work is partially supported by National Natural Science Foundation of China under grant number 12301531, Shandong Provincial Natural Science Fund for Excellent Young Scientists Fund Program (Overseas) under grant number 2023HWYQ-064, Shandong Provincial Youth Innovation Project under the grant number 2024KJN057 and  the OUC Scientific Research Program for Young Talented Professionals. 
	\bibliographystyle{abbrv}%{elsarticle-num-names}%{elsarticle-num}%{elsarticle-harv}%{plainnat}
	\bibliography{NSSAV}
	
\end{document}